\documentclass[12pt]{amsart}
\usepackage{amsmath,amsfonts,latexsym,graphicx,amssymb,url}
\usepackage{hyperref}
\usepackage{txfonts} 
\usepackage{amsmath,txfonts,pifont,bbding,pxfonts,manfnt}
\usepackage{psfrag}
\usepackage[active]{srcltx}
\usepackage{pdfsync}
\usepackage[none]{hyphenat}
\newcommand{\average}{{\mathchoice {\kern1ex\vcenter{\hrule
height.4pt width 6pt depth0pt} \kern-9.7pt}
{\kern1ex\vcenter{\hrule height.4pt width 4.3pt depth0pt}
\kern-7pt} {} {} }}

\begin{document}

\newcommand{\dist}{\text{dist}} 
\newcommand{\diam}{\text{diam}}
\newcommand{\trace}{\mathrm{trace}}
\newcommand{\R}{{\mathbb R}} 
\newcommand{\C}{{\mathbb C}}
\newcommand{\Z}{{\mathbb Z}}
\newcommand{\N}{{\mathbb N}}
\newcommand{\gradg}{\nabla_{\G}}
\newcommand{\e}{\varepsilon} 
\newcommand{\calS}{{\mathcal S}}
\newcommand{\calM}{{\mathcal M}}
\newcommand{\bH}{\mathrm{\bf H}}
\newcommand{\calL}{\mathcal L}
\newcommand{\calW}{{\mathrm W}}


\newcommand{\Rn}{\mathbb R^n}
\newcommand{\Rm}{\mathbb R^m}
\renewcommand{\L}[1]{\mathcal L^{#1}}
\newcommand{\G}{\mathbb G}
\newcommand{\U}{\mathcal U}
\newcommand{\M}{\mathcal M}
\newcommand{\eps}{\epsilon}
\newcommand{\BVG}{BV_{\G}(\Omega)}
\newcommand{\no}{\noindent}
\newcommand{\ro}{\varrho}
\newcommand{\rn}[1]{{\mathbb R}^{#1}}
\newcommand{\res}{\mathop{\hbox{\vrule height 7pt width .5pt depth 0pt
\vrule height .5pt width 6pt depth 0pt}}\nolimits}
\newcommand{\hhd}[1]{{\mathcal H}_d^{#1}} \newcommand{\hsd}[1]{{\mathcal
S}_d^{#1}} \renewcommand{\H}{\mathbb H}
\newcommand{\BVGL}{BV_{\G,{\rm loc}}}
\newcommand{\GH}{H\G}
\renewcommand{\diam}{\mbox{diam}\,}
\renewcommand{\div}{\mbox{div}\,}
\newcommand{\divg}{\mathrm{div}_{\G}\,}
\newcommand{\norm}[1]{\|{#1}\|_{\infty}} \newcommand{\modul}[1]{|{#1}|}
\newcommand{\per}[2]{|\partial {#1}|_{\G}({#2})}
\newcommand{\Per}[1]{|\partial {#1}|_{\G}} \newcommand{\scal}[3]{\langle
{#1} , {#2}\rangle_{#3}} \newcommand{\Scal}[2]{\langle {#1} ,
{#2}\rangle}
\newcommand{\fron}{\partial^{*}_{\G}}
\newcommand{\hs}[2]{S^+_{\H}({#1},{#2})} \newcommand{\test}{\mathbf
C^1_0(\G,\GH)} \newcommand{\Test}[1]{\mathbf C^1_0({#1},\GH)}
\newcommand{\card}{\mbox{card}}
\newcommand{\bom}{\bar{\gamma}}
\newcommand{\shpiu}{S^+_{\G}}
\newcommand{\shmeno}{S^-_{\G}}
\newcommand{\CG}{\mathbf C^1_{\G}}
\newcommand{\CH}{\mathbf C^1_{\H}}
\newcommand{\di}{\mathrm{div}_{X}\,}
\newcommand{\norma}{\vert\!\vert}

\theoremstyle{plain}
\newtheorem{theorem}{Theorem}[section]
\newtheorem{corollary}[theorem]{Corollary}
\newtheorem{lemma}[theorem]{Lemma}
\newtheorem{proposition}[theorem]{Proposition}
\newtheorem{definition}[theorem]{Definition}
\newtheorem{remark}[theorem]{Remark}
\providecommand{\bysame}{\makebox[3em]{\hrulefill}\thinspace}
\renewcommand{\theequation}{\thesection.\arabic{equation}}

\title[Interior second derivative  estimates]{interior second derivative  estimates for solutions to the linearized Monge--Amp\`ere equation}
\author[C.  E. Guti\'errez and T. Nguyen]{Cristian E. Guti\'errez  and  Truyen Nguyen}
\thanks{\today}
\thanks{CG  gratefully acknowledges the support  provided by NSF grant DMS-0901430}
\thanks{TN gratefully acknowledges the support provided by  NSF grant DMS-0901449}
\address{Cristian Guti\'errez , Temple University, Department of Mathematics, Philadelphia, PA 19122, USA.} 
\email{gutierre@temple.edu}
\address{Truyen Nguyen, The University of Akron, Department of Mathematics, Akron, OH 44325, USA.}
\email{tnguyen@uakron.edu}

\begin{abstract}
Let $\Omega\subset \R^n$ be a bounded convex domain and $\phi\in C(\overline\Omega)$  be a convex function such that $\phi$ is sufficiently smooth on $\partial\Omega$ and the  Monge--Amp\`ere measure $\det D^2\phi$ is bounded away from zero and infinity in $\Omega$.
 The corresponding linearized Monge--Amp\`ere equation is
\[
\trace(\Phi D^2 u) =f,
\]
where $\Phi := \det D^2 \phi ~ (D^2\phi)^{-1}$ is the matrix of cofactors of $D^2\phi$. We prove a conjecture in \cite{GT} about the relationship between $L^p$ estimates for $D^2 u$ and the closeness between $\det D^2\phi$ and one. As a consequence, we obtain interior $W^{2,p}$ estimates for solutions to such equation whenever the   measure $\det D^2\phi$ is given by a continuous density   and the function $f$ belongs to $L^q(\Omega)$ for some $q> \max{\{p,n\}}$.
\end{abstract}

\maketitle

\setcounter{equation}{0}
\section{Introduction}
%
%
%

 $L^p$-estimates play a fundamental role in the theory of second-order elliptic partial differential equations, with
many works devoted to the topic, see \cite[Chapter~9]{GiT} and \cite[Chapter~7]{CC}.
For linear equations of the form
$\text{trace}(A(x)D^2u(x))=f(x)$ in a domain $\Omega\subset \R^n$ with
\begin{equation}\label{uniformly-elliptic}
\lambda |\xi |^2\le\langle A(x)\xi ,\xi \rangle \le \Lambda |\xi
|^2\quad \mbox{for all $x\in \Omega$ and $\xi \in \R^n$},
\end{equation}
$L^p$-estimates for  second derivatives of solutions were derived in the
1950's as a consequence of the celebrated Calder\'on and Zygmund theory of singular
integrals. Precisely, if the matrix $A(x)$ is continuous in $\Omega$,
then for any domain $\Omega^\prime\Subset \Omega$ and any
$1<p<\infty$ we have
\begin{equation}\label{eq:calderonzymundinequality}
\|D^2u\|_{L^p(\Omega^\prime)}\le
C(\|u\|_{L^p(\Omega)}+\|f\|_{L^p(\Omega)}),
\end{equation}
where $C$ is a constant depending only on $p,\lambda, \Lambda, n, \text{dist}(\Omega^\prime,\partial \Omega)$ and the modulus of
continuity of $A(x)$. The continuity assumption on the coefficient matrix is essential  when $n\geq 3$.  Indeed, it is shown in \cite{U} and \cite{PT} that if $A(x)$ satisfies \eqref{uniformly-elliptic} and is merely
measurable, then \eqref{eq:calderonzymundinequality} is false for $p\geq 1$. However, it is proved in \cite{E, L} that estimates for second derivatives that do not depend on the continuity of $A(x)$ do hold when $p>0$ is sufficiently small.

$L^p$-estimates for second derivatives of solutions to fully nonlinear uniformly elliptic equations of the form $F(D^2 u,x) = f(x)$ were studied by 
Caffarelli \cite{C1}. In this  fundamental work, he established Calder\'on-Zygmund
type interior $W^{2,p}$ estimates for viscosity solutions under the assumptions that
$F(D^2 u, x)$ is suitably close to $F(D^2 u, 0)$, and solutions to the frozen equation $F(D^2 u, 0)=0$ admit interior $C^{1,1}$ estimates.  
For more details and related results to those of Caffarelli, we refer to \cite[Chapter~7]{CC} and \cite{Es,E, Sw, WL}.
By extending further his perturbation method in \cite{C1}, 
   Caffarelli \cite{C3} was able to derive  interior $W^{2,p}$ estimates for convex solutions to the Monge--Amp\`ere equation $\det D^2 \phi=g(x)$ under the optimal  condition that $g$ is continuous and bounded away from zero and infinity (see also \cite[Chapter 6]{G}, \cite{H, dPF} and the recent corresponding boundary estimates in \cite{S2}).

In this paper we consider the linearized Monge--Amp\`ere equation.
Let $\Omega\subset \R^n$ be a normalized convex domain and $\phi \in C(\overline \Omega)$ be a convex function satisfying 
$\lambda \leq \text{det}D^2\phi=g(x) \leq\Lambda$ in $\Omega$ and $\phi=0$ on $\partial\Omega$. The
linearized Monge--Amp\`ere equation corresponding to $\phi$ is
\begin{equation}\label{linear} 
\calL_{\phi}u :=\text{trace}(\Phi D^2u)=f(x)\quad\mbox{in}\quad\Omega
\end{equation} where $\Phi := (\text{det}D^2\phi)\, (D^2\phi)^{-1}$
is the matrix of cofactors of $D^2\phi$.
We note that  $\calL_\phi$ is both a non divergence and divergence differential operator which is  degenerate elliptic, that is, the matrix $\Phi(x)$ is positive semi-definite and does not satisfy \eqref{uniformly-elliptic}.
The equation \eqref{linear}  is of great importance as it appears in a number of problems. For example,
it appears in affine differential geometry in the solution of the affine Bernstein problem  (\cite{T, TW1, TW2, TW3, TW4}), and in the Aubreu's equation arising in the differential geometry of toric varieties (\cite{D1, D2, D3, D4,Z1, Z2}).
In addition, the equation appears in fluid mechanics in the semigeostrophic system which is an approximation to the incompressible Euler equation and is used in meteorology to study atmospheric flows (\cite{CNP, Lo}).
The linearized Monge--Amp\`ere equation was first studied by Caffarelli and Guti\'errez in \cite{CG2} where it is proved  
that nonnegative solutions to $\mathcal L_{\phi}u=0$ satisfy a uniform Harnack's inequality yielding, in particular, interior H\"older continuity of solutions. By using these interior  H\"older estimates and perturbation arguments,  we recently established in \cite{GN} Cordes--Nirenberg type interior $C^{1,\alpha}$  estimates for solutions to \eqref{linear}.

The purpose in this paper is to study the $L^p$ integrability of
second derivatives of solutions to  the   equation \eqref{linear}.
A previous result in this direction
is proved  by Guti\'errez and Tournier in \cite{GT}: for any domain $\Omega'\Subset\Omega$,
there exist $p>0$ small and $C>0$ depending only on
$\lambda, \Lambda, n$ and $\text{dist}(\Omega^\prime,\Omega)$ such
that
\begin{equation}\label{eq:mainestimate}
 \|D^2u\|_{L^p(\Omega^\prime)}\leq
C(\|u\|_{L^\infty(\Omega)}+\|f\|_{L^n(\Omega)})
\end{equation}
for all solutions $u\in C^2(\Omega)$ of  \eqref{linear}. Notice that since $\calL_\phi \phi = n \det D^2\phi =n g(x)$, it follows from Wang's counterexample \cite{W} that \eqref{eq:mainestimate} is false for any
$p>1$. In fact, if we hope the estimate \eqref{eq:mainestimate} to hold for large values of $p$, one needs to assume in addition that $g\in C(\Omega)$, see \cite[Section~8]{GT} for more details.
 In light of this, 
it was conjectured in \cite{GT} that the $L^p$-integrability of the second derivatives of $u$ in \eqref{linear} 
improves when $\det D^2\phi$ gets closer to one; in other words, if $1-\epsilon\leq \det D^2\phi\leq 1+\epsilon$, then the exponent $p=p(\epsilon)$ in \eqref{eq:mainestimate} satisfies $p(\epsilon)\to +\infty$ as $\epsilon\to 0$. 

In this article we solve the above conjecture in the affirmative, Theorem \ref{premainhomo}. As a consequence,
we obtain the following main result of the paper.
\begin{theorem}\label{mainhomo}
Let $\Omega$ be a normalized convex domain and $g\in C(\Omega)$ with $0<\lambda \leq g(x)\leq \Lambda$. Suppose $u\in W^{2,n}_{loc}(\Omega)$ is a solution of $\calL_\phi u=f$ in $\Omega$, where 
 $\phi\in C(\overline \Omega)$ is a convex function satisfying $\det D^2\phi =g$ in $\Omega$ and  $\phi=0$ on $\partial \Omega$. Let  $\Omega'\Subset \Omega$,  $p>1$ and $\max{\{n,p \}}<q <\infty$.
 Then there exists $C>0$ depending only on $p, q, \lambda, \Lambda, n, \dist(\Omega',\partial\Omega)$ and the modulus of continuity of $g$ such that
\begin{equation}\label{L^p-LMA}
\|D^2u\|_{L^p(\Omega')}\leq
C\left(\|u\|_{L^\infty(\Omega)} + \|f\|_{L^{q}(\Omega)}\right).
\end{equation}
\end{theorem}
The conditions on the Monge-Amp\`ere measure $\det D^2\phi$ are sharp  and  the constant in \eqref{L^p-LMA} depends on $\det D^2\phi$ and not on the maximum or minimum of eigenvalues of $D^2\phi$.
Our result can be viewed as a degenerate counterpart of the classical Calder\'on-Zygmund estimates \eqref{eq:calderonzymundinequality} for linear uniformly elliptic equations in non divergence form, and Caffarelli's interior $W^{2,p}$ estimates \cite[Theorem~1]{C1}, \cite[Chapter 7]{CC} for fully nonlinear uniformly elliptic equations.


 In order to address the lack of uniform ellipticity of the linearized Monge-Amp\`ere operator, we follow the strategy  in \cite{CG2} by working with  sections of solutions to the Monge-Amp\`ere equation.
The role that the sections play in our analysis is similar to that of Euclidean balls in the theory of uniformly elliptic equations.  In addition, to measure the degree of regularity of the solution we introduce the sets $G_M(u,\Omega)$ where the solution $u$ is touched by tangent paraboloids, see Definition~\ref{def:GMs}. In contrast with \cite[Definition~3.5]{GT}, the sets $G_M(u,\Omega)$ are now invariant by affine transformations. 
We note  that unlike the theory in  \cite[Chapter~7]{CC}, where the standard Euclidean distance is used,  our tangent paraboloids are defined with respect to a quasi distance induced by the solution $\phi$ of the Monge-Amp\`ere equation. With this new definition, our first step is to derive  rough  density estimates for the sets $G_M(u,\Omega)$ which are achieved by following the method in \cite{GT}. The next crucial step in solving the conjecture is to accelerate the initial density estimates. To  make this breakthrough, we use
a key idea introduced in \cite{GN}, that is, to compare solutions of two different linearized Monge-Amp\`ere equations. Precisely, we  compare solutions of $\calL_\phi u =f$ with solutions of $\calL_w h =0$, having the same Dirichlet boundary data, where $w$ is the solution of the Monge-Amp\`ere equation $\det D^2w=1$ in $\Omega$ and $w=0$ on $\partial\Omega$.
It is also important to know that the coefficient matrices of two different linearized equations are close in $L^p$-norm when the determinants of the corresponding convex functions are close in $L^\infty$-norm. This is given in our  recent work \cite{GN}.
These two comparison results allow us  to estimate explicitly $\|u-h\|_{L^\infty}$ in terms of $\|\det D^2\phi -1\|_{L^\infty}$, and by using this  approximation we can perform the acceleration process to obtain the necessary density estimates for the sets 
$G_M(u,\Omega)$. Finally and to conclude the proof of the conjecture, all these estimates permit us to use the covering theorems for sections of solutions to the Monge-Amp\`ere equation proved in \cite{CG1,CG2}.

\bigskip
To give more perspective, we mention the following recent work for the linearized Monge-Amp\`ere equation:
Sobolev type inequalities associated to the linearized operator $\calL_\phi$ (Tian and Wang \cite{TiW}), Liouville property for solutions to $\calL_\phi u =0$ in $\R^2$ ( Savin \cite{S1})
and 
 boundary $C^{1,\alpha}$ estimates  for $\calL_\phi u =f$ and its applications (Le and Savin \cite{LS1, LS2}).

The paper is organized as follows.
Section \ref{preliminaries} is devoted to preliminary results for solutions $\phi$ to the Monge-Amp\`ere equation 
that will be used later. We also introduce there a quasi metric and  the sets $G_M(u,\Omega)$ where the solution $u$ to the linearized equation is touched by tangent paraboloids associated to the quasi distance.
In Section \ref{sec:densitylemmas} we establish density estimates for the set $G_M(u,\Omega)$  and use them to derive the initial power decay for the distribution function giving small integrability  of $D^2u$.
Finally, Section \ref{sec:L^p-estimate} contains the main estimates in the paper showing how the integrability improves when 
$\det D^2\phi$ gets closer to one.

\section{Preliminary results}\label{preliminaries}

\subsection{Some properties for the Monge-Amp\`ere equation}
Given an open convex set $\Omega\subset\R^n$ and  a function $\phi\in C(\Omega)$, $\partial \phi$ denotes
the sub differential of $\phi$. The Monge-Amp\`ere measure
associated with $\phi$ is defined by
$M\phi(E) 
:= |\partial \phi (E)|$,  for all Borel subsets $E\subset \Omega$.
 The convex set $\Omega$ is called a {\it normalized convex domain} if $B_{1}(0) \subset \Omega
 \subset B_{n}(0)$. Here $B_R(y)$ denotes the Euclidean ball with radius $R$ centered at $y$.
Observe that by Fritz John's lemma, every bounded convex domain with non empty interior can be normalized, i.e.,
there is an invertible affine transformation $T$ with
$B_{1}(0) \subset T(S) \subset B_{n}
(0)$.
 A {\it section of a convex function $\phi\in C^1(\Omega)$} centered at $\bar x$ and with height $t$ is
defined by
\begin{equation*}\label{def:section}
S_\phi(\bar x, t) = \Big\{x\in \Omega: \phi(x)< \phi(\bar x) +  \nabla \phi(\bar x)\cdot (x-\bar x) + t\Big\}.
\end{equation*}
If $\phi=0$ on $\partial \Omega$, then for $0 <\alpha <1$ we set
\begin{equation}\label{def:omegaalpha}
\Omega_{\alpha}=\{x\in \Omega: \phi(x)<(1-\alpha
)\,\min_{\Omega}\phi \},
\end{equation}
and notice that $\Omega_\alpha$ is a section of $\phi$ at the minimum of $\phi$, i.e., $\Omega_\alpha=S_\phi(x_0,-\alpha\phi(x_0))$
where $x_0\in \Omega$ is such that  $\min_{\Omega}\phi=\phi(x_0)$.
We are going to list some basic  properties related to sections that will be used later.
All results in this subsection hold under the assumption:
\bigskip
\\
\noindent $(\bH)$ \ \ \ {\it $\Omega$ is a normalized convex domain  and $\phi \in C(\overline\Omega)$ is a convex function such that}
\[\lambda \leq M\phi \leq \Lambda \quad\mbox{in}\quad \Omega \quad \mbox{ and }\quad \phi =0 \quad 
\mbox{on}\quad
\partial\Omega.
\]
It is  known  from the works of Caffarelli \cite{C2, C4} that $\phi$ is strictly convex and $C^{1,\alpha}$ in the interior of $\Omega$. Moreover, we have the following lemma from \cite{GH} (see  \cite[Theorem~3.3.10]{G}).

\begin{lemma}\label{containment2}
Let  $0<\alpha < \beta<1$. Then for any $x\in \Omega_\alpha$, we have $S_\phi(
x,C_0(\beta -\alpha )^{\gamma}) \subset \Omega_\beta$ for some
$C_0$ and $\gamma$ depending only on $n$, $\lambda$ and $\Lambda$. Consequently,
there exists $\eta=\eta(\alpha,n,\lambda,\Lambda)>0$  such that $S_\phi ( x,t)
\Subset \Omega$ for all $x\in
\Omega_{\alpha}$ and $t\leq \eta$.
\end{lemma}

We now  state a result about strong type $p-p$ estimates for the maximal function with respect to sections. 
For that, let us fix  $0<\alpha_0<1$ and  take $\eta_0=\eta_0(\alpha_0,n, \lambda,\Lambda)$ be the corresponding positive constant given  by Lemma~\ref{containment2}.

\begin{theorem}\label{strongtype}
Let $\mu :=M\phi$ and define
\begin{equation*}\label{maximalfunction}
\mathcal M_\mu (f)(x) :=\sup_{t\leq\eta_0/2}
\dfrac{1}{\mu(S_\phi(x,t))}\int_{S_\phi(x,t)}|f(y)|\, d\mu (y)\quad \forall x\in \Omega_{\alpha_0}.
\end{equation*}
Then for any $1<p<\infty$, there exists a constant $C$ depending on $p$, $n$, $\lambda$ and $\Lambda$ such that
\[
\left(\int_{\Omega_{\alpha_0}}{|\mathcal M_\mu (f)(x)|^p ~d\mu(x)} \right)^{\frac{1}{p}}\leq
C\, \left(\int_{\Omega}{|f(y)|^p \,d\mu(y)} \right)^{\frac{1}{p}}.
\]
\end{theorem}
Notice that it is known from \cite{CG1} and \cite[Theorem~2.9]{GT} that $\mathcal M_\mu$ is of weak type $1-1$. This together with the obvious inequality $\|\mathcal M_\mu(f)\|_{L^\infty(\Omega_{\alpha_0})}\leq \|f\|_{L^\infty(\Omega)}$ and the Marcinkiewicz interpolation 
lemma (see Theorem~5 in \cite[Page~21]{St})  yields the strong type $p-p$ estimate in Theorem~\ref{strongtype}.
The next lemma is a slight modification of \cite[Lemma~6.2.1]{G}.
\begin{lemma}\label{lm:belowimplyabove}
There exist $c=c(n,\lambda,\Lambda)>0$ and $\delta_0= \delta_0(\alpha_0,n,\lambda,\Lambda)>0$ such that
if $x_0\in \Omega_{\alpha_0}$ and
$\phi(x)\geq
\phi(x_0) + \nabla \phi(x_0)\cdot (x-x_0) +\sigma\, |x-x_0|^2\quad\forall x\in \Omega$,
then 
\[\phi(x)\leq
\phi(x_0) + \nabla \phi(x_0)\cdot (x-x_0) +\frac{1}{c^2 \sigma^{n-1}}\, |x-x_0|^2
\quad \mbox{for all}\quad |x-x_0|\leq \delta_0.
\]
\end{lemma}
\begin{proof}
Let $u(x) := \phi(x)-
\phi(x_0)- \nabla \phi(x_0)\cdot (x-x_0)$. Then by the proof of Lemma~6.2.1 in \cite{G},  we have $u(x)\leq C(n,\lambda,\Lambda)\sigma^{-n+1}\, |x-x_0|^2$ for all $x\in\Omega$ satisfying $u(x)\leq \eta_0$. Next it follows from Aleksandrov's maximum principle and \cite[Proposition~3.2.3]{G} that
$\dist(\Omega_{\alpha_0}, \partial \Omega)\geq c(n,\lambda,\Lambda) (1-\alpha_0)^n =:d_{\alpha_0}$. Moreover if $\dist(x,\partial\Omega)\geq d_{\alpha_0}/2$, then  by using \cite[Lemma~3.2.1]{G} we obtain
\[
u(x) \leq |\nabla \phi(\xi) -\nabla\phi(x_0)| \, |x-x_0| \leq \frac{ C(n,\lambda,\Lambda)}{d_{\alpha_0}} |x-x_0|
\]
where $\xi$ is some point on the segment joining $x_0$ and $x$.
Therefore there exists $\delta_0>0$ depending only  on $\alpha_0$, $n$, $\lambda$ and $\Lambda$ such that $u(x)\leq \eta_0$ whenever $|x-x_0|\leq \delta_0$.
\end{proof}
The above  lemma together with Lemma~6.2.2 in \cite{G} gives:

\begin{lemma}\label{lm:excentricityofGM}
Given  $0<\alpha\leq \alpha_0 $ and  $\gamma>0$,  we define
\begin{equation}\label{def:newexcentricity}
D^\alpha_\gamma = \Big\{x\in \Omega_{\alpha}: S_\phi(x,t)\subset B_{\gamma\sqrt{t}}(x), \quad \forall t\leq \eta_0 \Big\}.
\end{equation}
Then there exist $c=c(n,\lambda,\Lambda)>0$ and $\delta_0= \delta_0(\alpha_0,n,\lambda,\Lambda)>0$ such that
for any $\gamma>0$ satisfying  $(c\gamma)^{\frac{1}{n-1}}\geq \dfrac{\diam(\Omega)}{\sqrt{\eta_0}}$, we have: 
 if $\bar x\in D^{\alpha}_{(c\gamma)^\frac{1}{n-1}}$ then 
\[\phi(x)
-\phi(\bar x) -
\nabla \phi(\bar x)\cdot (x-\bar x) \leq \gamma^2 |x-\bar
x|^2\quad \text{for all}\quad  |x-\bar x|\leq \delta_0.
\]
\end{lemma}
\begin{proof} Let $c$ and $\delta_0$ be given by Lemma~\ref{lm:belowimplyabove}, and take  $\bar x\in D^{\alpha}_{(c\gamma)^\frac{1}{n-1}}$. Since  $(c\gamma)^{\frac{1}{n-1}}\geq \frac{\diam(\Omega)}{\sqrt{\eta_0}}$, 
we then have by \cite[Lemma~6.2.2]{G} that $\bar x\in \Omega_{\alpha}\cap A_{(c\gamma)^{\frac{-2}{n-1}}}$,  where
\[
A_\sigma := \left\{x_0\in \Omega: \phi(x)\geq
\phi(x_0) + \nabla \phi(x_0)\cdot (x-x_0) + \sigma\, |x-x_0|^2, \quad \forall x\in \Omega\right\}.
\]
Therefore  the conclusion of the lemma follows from Lemma~\ref{lm:belowimplyabove}.
\end{proof}

\subsection{Tangent paraboloids and power decay for the Monge-Amp\`ere equation}
In this subsection we recall the quasi distance given by the convex function
 $\phi$ and then use it to define the sets where
the solution $u$ is touched from above and below by certain functions involving this quasi distance.

\begin{definition}\label{def:distance}
Let $\Omega$ be a bounded convex set in $\R^n$ and $\phi\in
C^1(\Omega)$ be a convex function.
For any  $x\in \Omega$ and $x_0\in \Omega$, we define $d(x, x_0)$ by
\begin{equation*}
d(x,x_0)^2 := \phi(x) - \phi(x_0)
-\nabla \phi(x_0)\cdot (x-x_0).
\end{equation*}
\end{definition}
Clearly $x\mapsto d(x,x_0)^2$ is a convex function on $\Omega$. Since $d(x,x_0)^2$ is in general not equivalent to $|x-x_0|^2$, the following definition of "tangent paraboloids" has  a nature different from the standard definition of tangent paraboloids for uniformly elliptic equations (see \cite{CC}). It is however more suitable to exploit the degenerate structure of the solution $\phi$ to the Monge-Amp\`ere equation. 
\begin{definition}\label{def:GMs}
Let $\Omega$ and $\phi$ be as in Definition~\ref{def:distance}.
Then for  $u\in C(\Omega)$ and $M>0$, we define the sets
\begin{align*}
&G_M^+(u,\Omega)\\
&= \big\{\bar x\in \Omega: \text{$u$ is differentiable at $\bar x$ and } u(x)\leq u(\bar x) +
\nabla u(\bar x)\cdot (x-\bar x) +
M\, d(x,\bar x)^2\,\forall x\in \Omega\big\};\\
&G_M^-(u,\Omega)\\
 &= \big\{\bar x\in \Omega:  \text{$u$ is differentiable at $\bar x$ and } u(x)\geq u(\bar x) +
\nabla u(\bar x)\cdot (x-\bar x) -
M\, d(x,\bar x)^2\, \forall x\in \Omega\big\};
\end{align*}
and
$G_M(u,\Omega) :=G_M^+(u,\Omega) \cap G_M^-(u,\Omega)$. 
\end{definition}
We note that local versions of Definition~\ref{def:distance} and Definition~\ref{def:GMs} were introduced in \cite{GT}. However, these definitions are not good enough for the purposes of this paper. The next observation is our starting point for deriving $L^p$-estimates  for second derivatives of solutions to the linearized Monge-Amp\`ere equation.

\begin{lemma}\label{distribution-function}
Assume condition $(\bH)$ holds. Let $0<\alpha\leq \alpha_0$, $u\in C(\Omega)$ and
\begin{equation}\label{function-theta}
\Theta(u)(x) := \frac{1}{2}\, \Theta(u,B_{\delta_0}(x))(x) \quad\mbox{for}\quad x\in \Omega_{\alpha_0}
\end{equation}
where $\delta_0$ is given by Lemma~\ref{lm:excentricityofGM} and  $\Theta(u,B_{\delta_0}(x))(x)$ is defined exactly as in \cite[Section~1.2]{CC}. Then for $\kappa>1$, we have
\begin{align}\label{eq:distributionsetofsecondderivatives}
\{x\in \Omega_\alpha: \Theta(u)(x) > \beta^\kappa\} 
\subset \big(\Omega_\alpha \setminus D^{\alpha}_{(c\beta^{\frac{\kappa-1}{2}})^{\frac{1}{n-1}}}\big) \cup \big(\Omega_\alpha
\setminus G_\beta(u,\Omega)\big)
\end{align}
for any $\beta>0$ satisfying $(c\beta^{\frac{\kappa -1}{2}})^{\frac{1}{n-1}}\geq \dfrac{\diam(\Omega)}{\sqrt{\eta_0}}$ with $c=c(n,\lambda,\Lambda)$ is as in Lemma~\ref{lm:excentricityofGM}.
\end{lemma}
\begin{proof}
Let $\gamma := \beta^{\frac{\kappa -1}{2}}$.
If $\bar x\in D^{\alpha}_{(c\gamma)^\frac{1}{n-1}} \cap G_\beta(u,\Omega)$, then
\[
-\beta \,d(x,\bar x)^2 \leq u(x)-u(\bar x) -\nabla u(\bar x)\cdot
(x-\bar x) \leq \beta \,d(x,\bar x)^2
\]
for each $x\in \Omega$. Since $\bar x\in
D^{\alpha}_{(c\gamma)^\frac{1}{n-1}}$, this together with  Lemma~\ref{lm:excentricityofGM} yields
\[
-\beta \gamma^2 |x-\bar x|^2 \leq u(x)-u(\bar x) -\nabla u(\bar x)\cdot
(x-\bar x) \leq \beta \gamma^2 |x-\bar x|^2
\]
for all $| x- \bar x|\leq \delta_0$, and so
$\Theta(u,B_{\delta_0}(x))(\bar x)\leq 2 \beta \gamma^2 = 2 \beta^{\kappa}$. Thus we have proved that
\[
D^{\alpha}_{(c\gamma)^\frac{1}{n-1}} \cap G_\beta(u,\Omega) \subset \{x\in \Omega_{\alpha}:
\Theta(u)(x) \leq \beta^\kappa\}
\]
and the lemma follows by taking complements.
\end{proof}

In order to derive interior $W^{2,p}$ estimates for solutions $u$  to the linearized Monge-Amp\`ere equation, we will  need to estimate the distribution function     
$F(\beta) :=| \{x\in \Omega_\alpha: \Theta(u)(x) > \beta^\kappa\} |$ for some suitable choice of $\kappa>1$. It follows from Lemma~\ref{distribution-function} that this can be done if one can get  appropriate fall off estimates for 
$F_1(\beta) := |\Omega_\alpha \setminus D^{\alpha}_{(c\beta^{\frac{\kappa -1}{2}})^\frac{1}{n-1}}|$ and $F_2(\beta) := |\Omega_\alpha
\setminus G_\beta(u,\Omega)|$ when $\beta$ is large. Notice  that since the function $F_1(\beta)$ involves only the solution $\phi$ of the Monge-Amp\`ere equation, its decay estimate has been established by Caffarelli in the fundamental work \cite{C3}. We reformulate his estimate  in the following theorem.

\begin{theorem}\label{thm:powerdecayMA}  Let $\Omega$ be a normalized convex domain and $\phi\in C(\overline{\Omega})$ be a convex function satisfying  $1-\epsilon\leq \det D^2\phi\leq 1+\epsilon$ in $\Omega$ and  $\phi=0$ on $\partial\Omega$, where  $0<\epsilon < 1/2$.
Then for any $0<\alpha<1$, there exists a positive constant $M$
depending only on $\alpha$ and $n$ such that
\begin{align}\label{eq:exponentialdecayofthesetD}
|\Omega_{\alpha}\setminus D^{\alpha}_{s}|
\leq \frac{|\Omega|}{(C_n\epsilon)^2}\quad s^{\dfrac{\ln \sqrt{C_n \epsilon}}{\ln
M}}\quad \mbox{ for all}\quad s\geq M.
\end{align}
\end{theorem}
\begin{proof}
This theorem is obtained by iterating \cite[Theorem~6.3.2]{G}. Indeed,
let $\alpha_0 := \frac{\alpha+1}{2}$ and 
let $M=M(\alpha_0,n)$ and $p_0 = p_0( \alpha_0, n)$ be the positive constants given by that theorem. By taking if necessary an even bigger constant, we can assume that $M$ is large so that $\alpha_0 - \sum_{j=1}^{\infty}{M^{-(j+1)p_0}}\geq 2\alpha_0-1$ and the statement of \cite[Theorem~6.3.2]{G} holds for all $\lambda\geq M$. We then begin the iteration with $\lambda=M$ and let $\alpha_1=\alpha_0- M^{-2 p_0}$. We get from \cite[Theorem~6.3.2]{G} that
\[
|\Omega_{\alpha_1}\setminus D^{\alpha_1}_{M^2}|\leq \sqrt{C_n\epsilon}~ |\Omega_{\alpha_0}\setminus D^{\alpha_0}_{M}|.
\] 
If $\lambda=M^2$ and $\alpha_2=\alpha_1- M^{-3 p_0}$, then
\[
|\Omega_{\alpha_2}\setminus D^{\alpha_2}_{M^3}|\leq \sqrt{C_n\epsilon}~|\Omega_{\alpha_1}\setminus D^{\alpha_1}_{M^2}|\leq \left(\sqrt{C_n\epsilon}\right)^2~ |\Omega_{\alpha_0}\setminus D^{\alpha_0}_{M}|.
\] 
Continuing in this way we let $\alpha_{k-1}=\alpha_{k-2}- M^{-k p_0}=\alpha_0- \sum_{j=1}^{k-1}{M^{-(j+1)p_0}}$
and obtain
\[
|\Omega_{\alpha_{k-1}}\setminus D^{\alpha_{k-1}}_{M^k}|\leq \left(\sqrt{C_n\epsilon}\right)^{k-1}~ |\Omega_{\alpha_0}\setminus D^{\alpha_0}_{M}|.
\]
Since by our choice of $M$, $\alpha_{k-1}\geq \alpha_0 - \sum_{j=1}^{\infty}{M^{-(j+1)p_0}}\geq 2\alpha_0-1$, it is easy to see that 
$\Omega_{2\alpha_0-1}\setminus D^{2\alpha_0-1}_{M^k} \subset \Omega_{\alpha_{k-1}}\setminus D^{\alpha_{k-1}}_{M^k}$.
Therefore, we have
\[
|\Omega_{\alpha}\setminus D^{\alpha}_{M^k}| = |\Omega_{2\alpha_0 -1 }\setminus D^{2\alpha_0-1}_{M^k}|\leq \left(\sqrt{C_n\epsilon}\right)^{k-1}~ |\Omega_{\alpha_0}\setminus D^{\alpha_0}_{M}|\quad \mbox{for}\quad k=1,2,\dots
\]
Now for each $s\geq M$, let us pick $k$
such that $M^k\leq s \leq M^{k+1}$. Then $D^{\alpha}_{M^k}\subset
D^{\alpha}_{s}\subset D^{\alpha}_{M^{k+1}}$ and $k\leq \log_M s \leq k+1$. So
\begin{align*}
|\Omega_{\alpha}\setminus D^{\alpha}_{s}|
&\leq \dfrac{|\Omega_{\alpha_0}|}{(C_n\epsilon)^2}\,(\sqrt{C_n \epsilon})^{k+1} 
\leq \dfrac{|\Omega_{\alpha_0}|}
{(C_n\epsilon)^2}\,(\sqrt{C_n \epsilon})^{\log_M s}
 =\dfrac{|\Omega_{\alpha_0}|}{(C_n\epsilon)^2}\quad s^{\frac{\ln \sqrt{C_n \epsilon}}{\ln
M}}.\nonumber
\end{align*}

\end{proof}

\section{$L^\delta$ estimates for second derivatives}\label{sec:densitylemmas}
In this section we prove two density lemmas and then use them to prove
a small power decay of $\mu (\Omega_\alpha \setminus G_\beta(u,\Omega))$
for $\beta$ large. Observe  that the density estimates established in \cite{GT} are not good enough for our purpose since
a different definition of the sets $G_{\beta}(u)$ was introduced there.   In \cite[Definition 3.5]{GT} the "tangent paraboloid" is assumed to lie below or above $u$ in a specific neighborhood depending on $\beta$  of the touching point. Such definition is not invariant under normalization and so not suitable for the acceleration process we consider later in Section~\ref{sec:L^p-estimate}. In this paper, we employ a global definition, Definition~\ref{def:GMs}, and we are still able to obtain  similar estimates as in \cite{GT} by modifying their arguments. For clarity, in the next subsection we give complete proofs of these estimates that are technically  simpler than the ones in \cite{GT}.
The following   lemma  is an    extension of \cite[Lemma~3.1]{GT} which allows us to work with strong solutions in $W^{2,n}_{loc}(\Omega)$ instead of classical solutions.

\begin{lemma}\label{lm:estimateofintegral}
 Let $\Omega\subset \R^n$ be open and  $u, \phi \in W^{2,n}_{loc}(\Omega)$ be such that $\det D^2\phi(x)>0$ for almost every $x$ in $ \Omega$. Let $w=u+\phi$.
Then for any Borel set $E\subset\Omega$, we have
\begin{equation}\label{eq:estimateofintegral}
M w(E)\leq \frac{1}{n^n}\,
 \int_{E\cap \mathcal C} \left( \Big( \dfrac{\trace\big(\Phi(x) D^2u(x)\big)}{\det D^2\phi(x)}+n
\Big)^+ \right)^n \,\det D^2\phi(x)\, dx
\end{equation}
where $\Phi(x)$ is the matrix of cofactors of $D^2\phi(x)$ and  $\mathcal C :=\{x\in \Omega:
w(x)=\Gamma(w)(x)\}$ with  $\Gamma(w)$ is the convex envelope of
$w$ in $\Omega$. 
\end{lemma}
\begin{proof}
Notice that the Sobolev embedding theorem guarantees that functions in $W^{2,n}_{loc}(\Omega)$ are continuous in $\Omega$. We first claim that
\begin{equation}\label{eq:subdifferential}
M w(F)  \leq
\int_{F} |\det D^2 w(x)| \,dx\quad \mbox{for all Borel sets}\quad F\subset \Omega.
\end{equation}
It is well known that \eqref{eq:subdifferential} holds if $w\in C^2(\Omega)$. For general $w\in W^{2,n}_{loc}(\Omega)$, let $\{w_m\}$ be a sequence of functions in $C^2(\Omega)$ converging to $w$ in the sense of $W^{2,n}_{loc}(\Omega)$. Let $U\subset \Omega$ be open and $K\subset  U$ be compact. Then $K\subset U\cap \Omega_{\epsilon}$ for all $\epsilon>0$ sufficiently  small, where $\Omega_{\epsilon} :=\{x\in \Omega:~ \dist(x,\partial\Omega)>\epsilon\}$. Since $Mw_m(U\cap \Omega_{\epsilon})  \leq
\int_{U\cap \Omega_{\epsilon}} |\det D^2 w_m(x)| \,dx$,  we get
\begin{align*}
\limsup_{m\to\infty}{M w_m(U\cap \Omega_{\epsilon})} 
\leq
\limsup_{m\to\infty}  \int_{U\cap \Omega_{\epsilon}} \Big|\det D^2 w_m - \det D^2 w\Big| \,dx 
+ \int_{U\cap \Omega_{\epsilon}} |\det D^2 w|\,dx.
\end{align*}
 Since the first term on the right hand side is clearly zero and the measures $M w_m$ converge to the measure $M w$ weakly, it follows by taking $\epsilon>0$ small enough that
$M w(K)
\leq
\int_{U} |\det D^2 w|\,dx$.
Consequently, 
\begin{equation}\label{weak-conver}
M w(U)
\leq \int_{U} |\det D^2 w(x)|\,dx
\end{equation}
by the regularity of the measure $Mw$.
Because \eqref{weak-conver} is true
 for any open set $U\subset\Omega$, we once again use the regularity of the measures to infer that the claim \eqref{eq:subdifferential} holds.

Now let $E\subset \Omega$ be an arbitrary Borel set. It is clear that $\partial w(E)=\partial w(E\cap \mathcal C)$ and so by using \eqref{eq:subdifferential} and the fact
$D^2w(x)\geq 0$ for almost every $x$ in   $\mathcal C$ we obtain
\begin{equation*}
M w(E) = M w(E\cap \mathcal C) \leq
\int_{E\cap \mathcal C} \det D^2 w(x)\,dx
\end{equation*}
and the estimate \eqref{eq:estimateofintegral}  follows by a calculation from \cite[Lemma~3.1]{GT}.
\end{proof}

Throughout this paper we always work with strong solutions, in the Sobolev space $W^{2,n}_{loc}(\Omega)$, of the linearized Monge-Amp\`ere equation. That is, the equation $\calL_\phi u=f$ in $\Omega$ is interpreted in the almost everywhere  sense in $\Omega$.

\subsection{Initial density estimates}
\begin{lemma}\label{lm:localcriticaldensityofGM}
Let $U$ be a normalized convex domain
and $\Omega$ be a bounded convex set such that $U\subset \Omega$. 
Let $\phi\in C^1(\Omega)\cap  W^{2,n}_{loc}(U)$ be a convex function satisfying 
$\lambda \leq \det D^2\phi \leq
\Lambda$ in $U$. Suppose $u\in  C(\Omega)\cap W^{2,n}_{loc}(U)\cap C^1(U)$, $ 0\leq u \leq 1$ in $\Omega$ and
$\calL_\phi u=f$ in $U$.
Then for each $\epsilon>0$ there exists  $\eta(\epsilon,n,\lambda, \Lambda)>0$ such
that for any $\eta\leq \eta(\epsilon,n,\lambda, \Lambda)$, we have 
\begin{align*}
\mu \big(G_{\frac{1}{\eta t_0}}^-(u,\Omega)\cap S_\phi(x_0,t_0)\big)
\geq\left[(1-\epsilon)^{\frac{1}{n}} -C \eta t_0\Big(\fint_{S_\phi(x_0,t_0)}{\big|\frac{f}{\det D^2\phi}\big|^n\,d\mu}\Big)^{\frac{1}{n}}  \right]^n\, \mu (S_\phi(x_0,t_0))
\end{align*}
for all sections $S_\phi(x_0,t_0)\Subset U$. Here  $\mu :=M\phi$ and $C$  depends only on $n$, $\lambda$ and $\Lambda$.
\end{lemma}

\begin{proof}
Let $T$ normalize the section $S_\phi(x_0,t_0)$. For $y\in T(\Omega)$, set
\begin{align*}
\tilde \phi(y)=
\frac{1}{t_0} \left[
\phi(T^{-1}y) -\phi(x_0) - \nabla \phi(x_0)\cdot (T^{-1} y-x_0) -t_0\right]
\quad\text{and}\quad  \tilde u(y)&=  u(T^{-1}y).
\end{align*}
We have that $\tilde \Omega:= T(S_\phi(x_0,t_0))$
is normalized and $\tilde \phi=0$ on $\partial \tilde \Omega$. Also, it follows from \cite[Lemma~2.3]{GN} that $\lambda'
\leq \det D^2\tilde \phi \leq \Lambda'$ in $\tilde \Omega$, where $\lambda'$ and $\Lambda'$ depend only on $n$, $\lambda$ and $\Lambda$. 
By Lemma~\ref{containment2}, for each $0<\alpha <1$ there
exists $\eta(\alpha)=\eta(\alpha,n,\lambda,\Lambda)$ such that if $\bar y\in \tilde
\Omega_\alpha$, then $S_{\tilde \phi} (\bar y, \eta(\alpha))\Subset 
\tilde \Omega$. Therefore if $\bar y
\in \tilde \Omega_\alpha$, then 
\begin{equation}\label{strictinclude}
\tilde \phi(\bar y) + \nabla \tilde \phi (\bar
y)\cdot (y-\bar y) +\eta(\alpha)<0 \quad\mbox{for all}\quad  y\in \overline{\tilde
\Omega}.
\end{equation}

Define $w_\alpha(y)= \eta(\alpha) \tilde u(y) +\tilde \phi(y)$. Let
$\gamma_\alpha$ be the convex envelope of $w_\alpha$ in $\tilde
\Omega$ and 
\begin{align*}
\mathcal C_\alpha =
\{ y\in \tilde \Omega: &\,w_\alpha(y)=\gamma_\alpha(y),\\
&\text{ and $\exists$ $\ell_y$ supporting hyperplane
to $\gamma_\alpha$ at $y$ with $\ell_y<0$ in $\tilde \Omega$}\}.
\end{align*}

\bf Claim 1. \rm $\nabla \tilde \phi(\tilde \Omega_\alpha)= \nabla (\tilde \phi +\eta(\alpha))(\tilde \Omega_\alpha) \subset \nabla
w_\alpha(\mathcal C_\alpha).$

To prove this claim, note first that  $\tilde \phi +\eta(\alpha)\geq w_\alpha$ in $\tilde\Omega$ and $w_\alpha\geq 0$  on $\partial\tilde\Omega$.
If $\bar y \in \tilde \Omega_\alpha$, then by \eqref{strictinclude} we know that the supporting plane $z= \tilde \phi (\bar y)+ \nabla \tilde \phi(\bar y)\cdot (y-\bar y) +\eta(\alpha)$ to $\tilde \phi +\eta(\alpha)$ at $\bar y$ has the property: $z<0$ on  $\partial\tilde\Omega$. Therefore, if we slide it down, then it must become a supporting plane to  $w_\alpha$ at some point $y^*\in \tilde\Omega$ (say $\ell_{y^*}$). Since  $z<0$ in  $\tilde\Omega$, so is $\ell_{y^*}$ and hence 
$y^*\in \mathcal C_\alpha$.
Thus  $\nabla \tilde \phi (\bar y)\in
\nabla w_\alpha (\mathcal C_\alpha)$ as desired.

\bf Claim 2. \rm $\mathcal C_\alpha \subset T\Big(
 G_{1/(t_0\eta(\alpha))}^-(u,\Omega) \cap S_\phi(x_0,t_0)\Big),$ for every
$0< \alpha<1$.

{\em Proof of Claim 2}. Let $\bar y\in \mathcal C_\alpha$. Then $\bar
y=T\bar x$ for some $\bar x\in S_\phi(x_0,t_0)$ and
\begin{equation}\label{eq:etapluspsibiggerthanell}
\eta(\alpha) \, \tilde u(y)+\tilde\phi(y)\geq \ell(y) \quad \forall  y\in
T(S_\phi(x_0,t_0))
\end{equation}
with equality at $y=\bar y$, for some $\ell$ affine with $\ell<0$
in $T(S_\phi(x_0,t_0))$. As
$\phi(x) =\phi(\bar x) +\nabla \phi(\bar x)\cdot (x-\bar x) + d(x,\bar x)^2$
 in $\Omega$,   
we have
\[
\tilde\phi(y)= \tilde\phi(\bar y) +\nabla \tilde\phi(\bar y)\cdot (y-\bar y) +
t_0^{-1}\, d(T^{-1}y,T^{-1}\bar y)^2 =: \tilde \ell_{\bar
y}(y) + t_0^{-1}\, d(T^{-1}y,T^{-1}\bar y)^2
\]
for all $y\in T(\Omega)$. This together with \eqref{eq:etapluspsibiggerthanell} gives
\[
\eta(\alpha) \,\tilde u(y) \geq \ell(y)-\tilde \ell_{\bar y}(y) -
t_0^{-1}\, d(T^{-1}y, T^{-1}\bar y)^2 =: g(y)\quad \forall y\in T(S_\phi(x_0,t_0))
\]
with equality at $y=\bar y$. 
Assume for a moment that
\begin{equation}\label{g<0}
0\geq g(y) \quad\mbox{for all}\quad  y\in T(\Omega)\setminus T(S_\phi(x_0,t_0)).
\end{equation}
Since  $u\geq 0$ in $\Omega$, we then obtain
\begin{equation}\label{eq:etatimesubiggerthang}
\eta(\alpha)\,\tilde u(y)\geq g(y) \quad \forall y\in T(\Omega).
\end{equation}
To see \eqref{g<0}, let
$B :=\{y\in T(\Omega): g(y)\geq 0\}$.
Note that $\bar y\in B$. Also $B\cap \partial
T(S_\phi(x_0,t_0))=\emptyset$ because if $y\in \partial T(S_\phi(x_0,t_0))$ then $\ell(y)<0$ and so
$g(y)<-\tilde \ell_{\bar y}(y) - t_0^{-1}\, d(T^{-1}y,
T^{-1}\bar y)^2 = -\tilde\phi(y)=0$. Moreover, $B$ is connected  as $g$ is concave. Hence $B\subset T(S_\phi(x_0,t_0))$  implying \eqref{g<0}.

 Since $\ell$ is a supporting hyperplane to $\eta(\alpha)\,\tilde u(y) + \tilde\phi(y)$ at
$\bar y$, and $u,\phi\in C^1(U)$, it follows from \eqref{eq:etatimesubiggerthang} that
\[
\tilde u(y) \geq \tilde u(\bar y) + \nabla \tilde u(\bar y)\cdot
(y-\bar y) -\dfrac{1}{t_0\eta(\alpha)} \,d(T^{-1}y, T^{-1}\bar
y)^2,\quad \forall y\in T(\Omega).
\]
Thus we have proved that 
\begin{align*}
\mathcal C_\alpha \subset
\{\bar y\in T(S_\phi(x_0,t_0)):
\tilde u(y)\geq \tilde u(\bar y) + \nabla \tilde u(\bar
y)\cdot (y-\bar y) -
\dfrac{1}{t_0\eta(\alpha)} d(T^{-1}y, T^{-1}\bar y)^2
\, \forall y\in T(\Omega)\}
\end{align*}
yielding Claim~2 because  $\tilde u(y)=u(T^{-1}y)$.

Now let $\tilde \Phi(y) :=(D^2\tilde \phi (y))^{-1} \det
D^2\tilde \phi (y)$. Then as $D^2\tilde \phi(y)= t_0^{-1} (T^{-1})^t\,
D^2\phi(T^{-1}y) \, T^{-1}$ and  $D^2\tilde u(y)= (T^{-1})^t\, D^2u(T^{-1}y) \, T^{-1}$, we get
\[
\trace (\tilde \Phi(y) \,D^2\tilde
u(y)) =\frac{t_0}{t_0^n |\det T|^2}\,\trace (\Phi(T^{-1}y) \,D^2
u(T^{-1}y)) = \frac{t_0\, f(T^{-1} y)}{t_0^n |\det T|^2}\quad\mbox{in}\quad \tilde \Omega. 
\]
Therefore by applying Lemma \ref{lm:estimateofintegral}
with $\Omega\rightsquigarrow \tilde\Omega$, $u\rightsquigarrow \eta(\alpha)\tilde u$,
$\phi\rightsquigarrow \tilde \phi$, $E=\mathcal C_\alpha$ and using Claim 1 and the fact $t_0^n  |\det T|^2 \approx 1$,  we obtain
\begin{equation*}
\int_{\tilde \Omega_\alpha} \det D^2 \tilde \phi(y)\,dy \leq
\frac{1}{n^n}\int_{\mathcal C_\alpha}\left(\frac{C\eta(\alpha)\, t_0\, |f(T^{-1} y)|}{\det D^{2} \tilde\phi(y)} +n\right)^n \det D^2 \tilde \phi(y)\,dy.
\end{equation*}
Since $\mathcal C_\alpha \subset T\left( G_{1/
(t_0\eta(\alpha))}^-(u,\Omega)\cap S_\phi(x_0,t_0) \right)$  by Claim 2, $\det D^2\tilde \phi(y)= \det D^2
\phi(T^{-1}y)
$ and $\tilde \Omega_\alpha=T(S_\phi(x_0,\alpha t_0))$,  the above inequality implies 
\begin{equation*}
\int_{S_\phi(x_0,\alpha t_0)} \det D^2 \phi(x)\,dx \leq
\frac{1}{n^n}\int_{G_{1/
(t_0\eta(\alpha))}^-(u,\Omega)\cap S_\phi(x_0,t_0)}{\left(\frac{C \eta(\alpha)\,t_0\, |f(x)|}{\det D^{2} \phi(x)} +n\right)^n \det D^2 \phi(x)\,dx.}
\end{equation*}
We then infer from Minkowski's inequality and $\mu =M \phi$ that
\begin{align*}
&\mu(S_\phi(x_0,\alpha t_0)^{\frac{1}{n}}\\
&\leq \frac{C \eta(\alpha) t_0}{n}\left(\fint_{S_\phi(x_0,t_0)}{\Big|\frac{f}{\det D^2\phi}\Big|^n\,d\mu}\right)^{\frac{1}{n}} \mu(S_\phi(x_0,t_0))^{\frac{1}{n}}
 +\mu\left(G_{1/(t_0\eta(\alpha))}^-(u,\Omega)\cap S_\phi(x_0,t_0)\right)^{\frac{1}{n}}.
\end{align*}
Note that this inequality holds for any $\eta\leq \eta(\alpha)$. 
Given $\epsilon>0$ there exists
$\alpha=\alpha(\epsilon)$ sufficiently close to one such that $
(1-\epsilon)\mu(S_\phi(x_0, t_0)) \leq \mu(S_\phi(x_0, \alpha
t_0))$,
which combined with the previous inequality yields the lemma for any $\eta\leq \eta(\alpha(\epsilon))$.
\end{proof}

In the next lemma, we no longer require $0\leq u \leq 1$ in $\Omega$ as in Lemma~\ref{lm:localcriticaldensityofGM}.

\begin{lemma}\label{lm:localcriticaldensitytwo}
Let $U$ be a normalized convex domain and $\Omega$ be a bounded convex  set such that $U\subset\Omega$. Let $\phi\in  C^1(\Omega)\cap  W^{2,n}_{loc}(U)$ be a convex function satisfying  $\lambda \leq \det D^2\phi \leq
\Lambda$ in $U$. Suppose $u\in  C(\Omega)\cap W^{2,n}_{loc}(U)\cap C^1(U)$
is a solution of  $\calL_\phi u=f$ in $U$. Then for each $\epsilon>0$ there exists $\eta(\epsilon,n,\lambda,\Lambda)>0$ such
that if $S_\phi(x_0,t_0)\Subset U$ and 
$S_\phi(x_0,t_0)\cap G_\gamma(u,\Omega)$ contains a point $\bar x$ with $S_\phi(\bar x,\theta t_0)\Subset U$, then we have
\begin{align*}
\mu \big(G_{\frac{2\theta\gamma}{\eta}}^-(u,\Omega)\cap S_\phi(x_0,t_0)\big)
\geq\left[(1-\epsilon)^{\frac{1}{n}} -\frac{\eta}{2n\theta\gamma}\Big(\fint_{S_\phi(x_0,t_0)}{\big|\frac{f}{\det D^2\phi}\big|^n d\mu}\Big)^{\frac{1}{n}} \right]^n\, \mu (S_\phi(x_0,t_0))
\end{align*}
for all $\eta\leq \eta(\epsilon,n,\lambda,\Lambda)$. Here $\mu :=M\phi$ and   $\theta=\theta(n,\lambda,\Lambda)>1$ is the engulfing constant given by \cite[Theorem~3.3.7]{G}.
\end{lemma}
\begin{proof}
Let $T$ normalize
$S_\phi(x_0,t_0)$, and for $y\in T(\Omega)$ we set
\begin{align*}
\tilde\phi(y) = \dfrac{1}{t_0} \Big[
\phi(T^{-1}y)  -  \phi(x_0)
- \nabla \phi(x_0)\cdot (T^{-1}y -x_0) -t_0\Big]\quad
\text{and}\quad  \tilde u(y)&= \dfrac{1}{2\theta t_0}\, u(T^{-1}y).
\end{align*}
It follows that $\tilde\Omega :=T(S_\phi(x_0,t_0))$ is normalized, $\tilde\phi=0$ on $\partial \tilde \Omega$ and $\lambda'
\leq \det D^2\tilde\phi \leq \Lambda'$ in $\tilde \Omega$, where $\lambda'$ and $\Lambda'$ depend only on $n$, $\lambda$ and $\Lambda$. 

Let $\bar x\in S_\phi(x_0,t_0)\cap G_\gamma(u,\Omega)$ be  such that $S_\phi(\bar x,\theta t_0)\Subset U$,
and define $\bar y=T\bar x$. Then
$-\gamma \, d(x,\bar x)^2 \leq u(x)-u(\bar x) - \nabla u(\bar
x)\cdot (x-\bar x) \leq \gamma \, d(x,\bar x)^2$ for all $x$ in $\Omega$.
Hence by changing variables we get
\begin{equation}\label{eq:boundedbydistance}
-\gamma \, \dfrac{d(T^{-1}y,T^{-1}\bar y)^2}{2\theta t_0} \leq \tilde
u(y)- \tilde u(\bar y) - \nabla \tilde u(\bar y)\cdot (y-\bar y)
\leq \gamma \, \dfrac{d(T^{-1}y,T^{-1}\bar y)^2}{2\theta t_0}, \forall y\in T(\Omega).
\end{equation}
Since $\bar x\in S_\phi(x_0,t_0)$, we have $S_\phi(x_0,t_0)\subset
S_\phi(\bar x, \theta t_0)$ by the engulfing property. So, if
$x\in S_\phi(x_0,t_0)$, then $d(x,\bar x)^2 \leq \theta t_0$, and
consequently $d(T^{-1}y,T^{-1}\bar y)^2\leq \theta t_0$ for all
$y\in T(S_\phi(x_0,t_0))$. This together with \eqref{eq:boundedbydistance} gives
\[
-\frac{\gamma}{2} \leq \tilde u(y)- \tilde u(\bar y) - \nabla \tilde
u(\bar y)\cdot (y-\bar y) \leq \frac{\gamma}{2} \quad\mbox{in}\quad \tilde\Omega.
\]
Hence if $
v(y) : = \dfrac{1}{\gamma}\left[\tilde u(y)- \tilde u(\bar y) - \nabla \tilde u(\bar
y)\cdot (y-\bar y)+\frac{\gamma}{2}\right]$ for $y\in T(\Omega)$,
then $0\leq v \leq 1$ in $\tilde \Omega.$

Let $0<\alpha <1$. There exists $\eta(\alpha)=\eta(\alpha,n,\lambda,\Lambda)>0$ such that if $\bar
y\in \tilde \Omega_\alpha$, then
\begin{equation}\label{secondstrictinclude}
\tilde\phi(\bar y) + \nabla \tilde\phi(\bar y) \cdot (y-\bar y) + \eta(\alpha)
\big(\theta +\dfrac12\big) <0, \quad \text{
for all $y\in \overline{\tilde \Omega}$}.
\end{equation}

Define $w_\alpha(y)= \eta(\alpha) v(y) + \tilde\phi(y)$.  Let $\gamma_\alpha$ be the convex envelope of $w_\alpha$
in $\tilde \Omega$, and
\begin{align*}
&\mathcal C_\alpha = \Big\{ \tilde y\in \tilde \Omega: w_\alpha(\tilde
y)=\gamma_\alpha(\tilde y), \, \text{and $\exists$ $\ell$ supporting hyperplane to
$\gamma_\alpha$ at $\tilde y$,}\\
&\qquad\qquad\qquad\qquad\qquad\qquad\qquad\qquad\qquad\quad \text{with $\ell <-\eta(\alpha)(\theta -\dfrac12)$ in
$\tilde \Omega$}\Big\}.
\end{align*}

\bf Claim 1. \rm $\nabla \tilde\phi(\tilde \Omega_\alpha) =\nabla (\tilde\phi+ \eta(\alpha))(\tilde \Omega_\alpha)\subset \nabla
w_\alpha(\mathcal C_\alpha).$

The proof of this is similar to that of Claim~1 in Lemma~\ref{lm:localcriticaldensityofGM}.

\bf Claim 2. \rm $\mathcal C_\alpha \subset T\left(G_{2\theta\gamma/\eta(\alpha)}^-(u,\Omega)\cap  S_\phi(x_0,t_0) \right)$ for every $0<\alpha<1$.

{\em Proof of Claim 2}. Let $\tilde y\in \mathcal C_\alpha$. There
exists $\ell$ affine such that $\eta(\alpha)v(y)+ \tilde\phi(y)\geq
\ell(y)$ for all $y\in T(S_\phi(x_0,t_0)),$ and with equality at
$y=\tilde y$, and $\ell<-\eta(\alpha) (\theta -1/2)$ in
$T(S_\phi(x_0,t_0))$. Since 
$\tilde\phi(y) = \tilde \ell_{\tilde y}(y) + \dfrac{1}{t_0}\,
d(T^{-1}y, T^{-1}\tilde y)^2$
 $\forall y\in T(\Omega)$  where $\tilde \ell_{\tilde y}(y) := \tilde\phi (\tilde y) +\nabla \tilde\phi(\tilde
y)\cdot (y-\tilde y)$, we then have
\begin{equation}\label{eq:etapsibiggerthandiferenceofells}
\eta(\alpha)v(y) \geq \ell(y)-\tilde \ell_{\tilde y}(y) -
\dfrac{1}{t_0}\, d(T^{-1}y, T^{-1}\tilde y)^2 =:g(y) \quad \forall  y\in T(S_\phi(x_0,t_0))
\end{equation}
with equality at $y=\tilde
y$. {\it Our goal is to extend 
\eqref{eq:etapsibiggerthandiferenceofells} to the set
$T(\Omega)$}. We claim that
\begin{equation}\label{g<constant}
g(y) < \dfrac{\eta(\alpha)}{2}\, \left( 1 -
\dfrac{ d(T^{-1} y, T^{-1}\bar y)^2}{t_0\theta} \right)  \quad \forall y\in  T(\Omega)\setminus
T(S_\phi(x_0,t_0)).
\end{equation}
Assume this claim for a moment.
Notice that from \eqref{eq:boundedbydistance} we have that
\begin{equation}\label{etav}
\eta(\alpha)
v(y) \geq \dfrac{\eta(\alpha)}{2}\, \left( 1 -
\dfrac{ d(T^{-1} y, T^{-1}\bar y)^2}{t_0\theta} \right)  \quad \forall y\in T(\Omega),
\end{equation}
and therefore \eqref{eq:etapsibiggerthandiferenceofells}  holds for
all $y\in T(\Omega)$.
Using 
\eqref{eq:etapsibiggerthandiferenceofells}, the fact
$g(y)= \eta(\alpha) v(\tilde y) + \eta(\alpha) \nabla v(\tilde
y)\cdot (y-\tilde y) -
\dfrac{1}{t_0}\, d(T^{-1} y, T^{-1}\tilde y)^2$
and the definition of $v$, we obtain
\[
\tilde u(y)\geq \tilde u(\tilde y) + \nabla \tilde u(\tilde
y)\cdot (y-\tilde y) - \dfrac{\gamma}{t_0 \eta(\alpha)}\,
d(T^{-1} y, T^{-1}\tilde y)^2 \quad \forall y \in T(\Omega).
\]
Thus we have shown that
\begin{align*}
\mathcal C_\alpha
&\subset \{\tilde y\in T(S_\phi(x_0,t_0)): \tilde u(y)\geq \tilde
u(\tilde y) + \nabla \tilde u(\tilde y)\cdot (y-\tilde y) -
\dfrac{\gamma}{t_0 \eta(\alpha)}\,
d(T^{-1} y, T^{-1}\tilde y)^2 ~\forall y\in T(\Omega)\}\\
&= T\Big\{\tilde x\in S_\phi(x_0,t_0): u(x)\geq u(\tilde x) + \nabla
u(\tilde x)\cdot (x-\tilde x) - \dfrac{2
\theta \gamma}{\eta(\alpha)}\,
d(x, \tilde x)^2\quad \forall x\in \Omega\Big\}\\
&=T\Big(G_{\frac{2
\theta \gamma}{\eta(\alpha)}}^-(u,\Omega)\cap S_\phi(x_0,t_0)\Big).
\end{align*}
So Claim 2 holds as long as \eqref{g<constant} is proved. Observe that \eqref{g<constant} is equivalent to
\begin{equation}\label{eq:BcontainedinTS}
 B := \left\{y\in T(\Omega): g( y) \geq
\dfrac{\eta(\alpha)}{2}\, \left(1 - \dfrac{d(T^{-1} y, T^{-1}\bar y)^2}{t_0\theta}\right)\right\} \subset T(S_\phi(x_0,t_0)).
\end{equation}
Since $\eta(\alpha)/2\theta <1$, we have that the function
\begin{align*}
&-g(y)+\dfrac{\eta(\alpha)}{2}\, \left(1 - \dfrac{d(T^{-1} y, T^{-1}\bar y)^2}{t_0\theta}\right)\\
&=-\ell(y) +\tilde \ell_{\tilde y}(y) +
\dfrac{1}{t_0}\, d(T^{-1}y, T^{-1}\tilde y)^2 + \dfrac{\eta(\alpha)}{2}\, \left(1 - \dfrac{d(T^{-1} y, T^{-1}\bar y)^2}{t_0\theta}\right)\\
&=-\ell(y)+\frac{\eta(\alpha)}{2} +  \frac{\eta(\alpha)}{2\theta}\big[ \tilde\phi(\bar y) +\nabla\tilde\phi(\bar y)\cdot (y-\bar y)\big] + \left(1- \frac{\eta(\alpha)}{2\theta}\right) \tilde\phi(y)
\end{align*}
is convex and hence $B$ is connected. Moreover, $\tilde y \in B\cap T(S_\phi(x_0,t_0))$ by 
 \eqref{etav} and since $g(\tilde y) = \eta(\alpha) v(\tilde y)$. Thus, \eqref{eq:BcontainedinTS} will follow if
\begin{equation}\label{eq:Btildedoesnotcut}
 B\cap \partial T(S_\phi(x_0,t_0))=\emptyset.
\end{equation}
Recall that $\ell < -\eta(\alpha) (\theta -1/2)$ in
$T(S_\phi(x_0,t_0))$, and $\tilde \ell_{\tilde y}(y) +
\dfrac{1}{t_0} \,d(T^{-1} y, T^{-1}\tilde y)^2 = \tilde\phi(y)=0$ on $\partial T(S_\phi(x_0,t_0))$. 
In addition, $ d(T^{-1}y, T^{-1}\bar  y)^2\leq \theta t_0$ in
$T(S_\phi(x_0,t_0))$  since $S_\phi(x_0, t_0)\subset S_\phi(\bar x, \theta t_0)$.
Therefore, if 
$y\in \partial T(S_\phi(x_0,t_0))$  then 
\begin{align*}
&-g(y)+\dfrac{\eta(\alpha)}{2}\, \left(1 - \dfrac{d(T^{-1} y, T^{-1}\bar y)^2}{t_0\theta}\right)\\
&=-\ell(y) +\tilde \ell_{\tilde y}(y) + \dfrac{1}{t_0}\, d(T^{-1}y, T^{-1}\tilde y)^2 + \dfrac{\eta(\alpha)}{2}\, \left(1 - \dfrac{d(T^{-1} y, T^{-1}\bar y)^2}{t_0\theta}\right)
\geq \eta(\alpha)\big(\theta - \frac{1}{2}\big)>0,
\end{align*}
and hence \eqref{eq:Btildedoesnotcut} holds as desired. 
This completes the proof of \eqref{g<constant}, and so 
Claim~2 is  proved. 

The lemma now  follows by applying 
Lemma~\ref{lm:estimateofintegral}
with $\Omega\rightsquigarrow \tilde\Omega$, $u\rightsquigarrow \eta(\alpha)v$,
$\phi\rightsquigarrow \tilde\phi$, $E=\mathcal C_\alpha$ and using Claim~1 and Claim~2. The detailed calculations are the same as those in Lemma~\ref{lm:localcriticaldensityofGM}.
\end{proof}

\subsection{Initial power decay for the linearized Monge-Amp\`ere equation}
We next use Lemma~\ref{lm:localcriticaldensityofGM} and Lemma~\ref{lm:localcriticaldensitytwo} to derive  a small power decay estimate. To achieve this,  the covering result proved in \cite{CG2} is essential.

 \begin{proposition}\label{Initial-Estimate}
 Let $U$ be a normalized convex domain and $\Omega$ be a bounded convex set such that $U\subset \Omega$. Let $\phi\in   C^1(\Omega)\cap  W^{2,n}_{loc}(U)$
be a convex function satisfying  $\lambda \leq \det D^2 \phi
\leq \Lambda$ in $U$ and $\phi=0$ on $\partial U$. Suppose $u\in  C(\Omega)\cap W^{2,n}_{loc}(U)\cap C^1(U)$, $|u|\leq 1$ in $\Omega$ and 
$\calL_\phi u=f$ in $U$ with
$\|f/\det D^2 \phi \|_{L^n(U,\mu)}\leq 1$.
Then for any $0<\alpha<1$, there exist $C,\,\tau>0$ depending only on $\alpha$, $n$, $\lambda$ and $\Lambda$ such that
\begin{equation*}
\mu\big(U_{\alpha}\setminus G_{\beta}(u,\Omega)\big) \leq
\dfrac{C}{\beta^\tau}\quad \mbox{for all $\beta$ large},
\end{equation*}
where $U_\alpha$ is defined as in  \eqref{def:omegaalpha}.
\end{proposition}
\begin{proof}
Let $0<\epsilon <1/2$ and  $\eta(\epsilon,n,\lambda, \Lambda)$ be the smallest of the constants in
Lemma~\ref{lm:localcriticaldensityofGM} and
Lemma~\ref{lm:localcriticaldensitytwo}.
Next fix $0<\eta\leq \eta(\epsilon,n,\lambda, \Lambda)$  small so that 
$[(1-\epsilon)^{1/n} -C \eta ]^n
\geq 1-2 \epsilon$, where $C=C(n,\lambda,\Lambda)$.
Applying Lemma
\ref{lm:localcriticaldensityofGM} to the functions
$\dfrac{u+1}{2}$ and $\dfrac{-u+1}{2}$, and noticing that
$G_N^-(\frac{u+1}{2},\Omega)= G_{2N}^-(u,\Omega)$ and $G_N^-(\frac{-u+1}{2})=
G_{2N}^+(u,\Omega)$, we obtain 
\begin{align*}
&\mu \big(S_\phi(x_0,t_0)\cap G_{2/\eta t_0}^-(u,\Omega)\big)
\geq
\left[(1-\epsilon)^{1/n} -C \eta \right]^n\, \mu \big(S_\phi(x_0,t_0))\geq (1-2 \epsilon)\, \mu (S_\phi(x_0,t_0)),\\
&\mu (S_\phi(x_0,t_0)\cap G_{2/\eta t_0}^+(u,\Omega)\big)
\geq
\left[(1-\epsilon)^{1/n} -C \eta \right]^n\, \mu (S_\phi(x_0,t_0))\geq (1-2 \epsilon)\, \mu (S_\phi(x_0,t_0))
\end{align*}
for any $S_\phi(x_0,t_0)\Subset U$.  
Taking $M:=2\theta/\eta$, it then follows that
\begin{align*}
\mu \Big(S_\phi(x_0,t_0)\setminus G_{M/\theta t_0}(u,\Omega)\Big)
&\leq \mu \Big(S_\phi(x_0,t_0)\setminus G_{M/\theta t_0}^+(u,\Omega) \Big) +
\mu \Big(S_\phi(x_0,t_0)\setminus
G_{M/\theta t_0}^-(u,\Omega)  \Big)\\
&\leq 4\epsilon \mu(S_\phi(x_0,t_0))
\end{align*}
as long as $S_\phi(x_0,t_0)\Subset U$.

Set $\alpha_0 := \frac{\alpha +1}{2}$. Assume $\alpha_2<\alpha_1<\alpha_0$ are such that there exist $\eta_2<\eta_1$ with the property:  
if $x\in U_{\alpha_2}$ and $t\leq \eta_2$ then
$S_\phi(x,t)\subset U_{\alpha_1}$; and if $x\in U_{\alpha_1}$ and $t\leq \eta_1$ then
$S_\phi(x,t)\subset U_{\alpha_0}$. 
Let $h\geq 1/\eta_0$ satisfy $1/\theta h \leq \eta_2$. For $x_0\in
U_{\alpha_2}\setminus G_{h M}(u,\Omega)$, 
define
\[
g(t) := \dfrac{\mu\Big((U_{\alpha_2}\setminus G_{h M}(u,\Omega))
\cap S_\phi(x_0,t)\Big)}{\mu(S_\phi(x_0,t))},\quad t>0.
\]
We have $\lim_{t\to 0} g(t)=1.$ Also, if $1/\theta h \leq t
<\eta_1$, then $S_\phi(x_0,t)\subset U_{\alpha_0}$ and
\begin{align*}
\mu\Big((U_{\alpha_2}\setminus G_{h M}(u,\Omega)) \cap
S_\phi(x_0,t)\Big) &\leq
\mu(S_\phi(x_0,t)\setminus G_{h M}(u,\Omega))\\
&\leq
\mu(S_\phi(x_0,t)\setminus G_{M/\theta t}(u,\Omega))
\leq 4\epsilon \, \mu(S_\phi(x_0,t)),
\end{align*}
since $G_{M/\theta t}(u,\Omega)\subset G_{h M}(u,\Omega)$. Therefore $g(t)\leq
4\epsilon$ for $t\in [1/\theta h, \eta_1)$ and so by continuity of
$g$, there exists $t_{x_0}\leq 1/\theta h $ satisfying  $g(t_{x_0})=4\epsilon$. Thus, 
we have shown that  for any $x_0\in
U_{\alpha_2}\setminus G_{h M}(u,\Omega)$ there is
$t_{x_0}\leq 1/\theta h$ such that
\begin{equation}\label{eq:criticaldensitynonhomoge}
\mu\Big((U_{\alpha_2}\setminus G_{h M}(u,\Omega)) \cap
S_\phi(x_0,t_{x_0})\Big) = 4\epsilon \, \mu(S_\phi(x_0,t_{x_0})).
\end{equation}
We now claim that \eqref{eq:criticaldensitynonhomoge} implies 
\begin{equation}\label{eq:sectioncontainedinomegaunionmaxfunctionbig}
S_\phi(x_0,t_{x_0}) \subset \big(U_{\alpha_1}\setminus G_h
(u,\Omega)\big) \cup \big\{x\in U_{\alpha_0}: \mathcal M_\mu ((f/\det D^2
\phi)^n)(x)> (2\theta n h)^n\big\}.
\end{equation}
Otherwise, and since $x_0\in U_{\alpha_2}$ and $t_{x_0}\leq
1/\theta h \leq \eta_2$, we have $S_\phi(x_0,t_{x_0})\subset
U_{\alpha_1}$ and there exists $\bar x\in
S_\phi(x_0,t_{x_0})\cap G_h(u,\Omega)$ such that $\mathcal M_\mu
((f/\det D^2 \phi)^n)(\bar x)\leq (2\theta n h)^n$. Note also that $S_\phi(\bar x,\theta t_{x_0}) \Subset U$ as $\bar x\in U_{\alpha_0}$ and $\theta t_0 \leq 1/h\leq \eta_0$. Then by
Lemma~\ref{lm:localcriticaldensitytwo} applied to $u$ and $-u$ and by our choice of $\eta$, we obtain
\begin{align*}
&(1-2 \epsilon) \, \mu(S_\phi(x_0,t_{x_0})) <
\mu(S_\phi(x_0,t_{x_0})\cap G_{h M}^-(u,\Omega)),\\
&(1-2\epsilon) \, \mu(S_\phi(x_0,t_{x_0})) <
\mu(S_\phi(x_0,t_{x_0})\cap G_{h M}^+(u,\Omega)).
\end{align*}
Hence
\[
\mu\Big((U_{\alpha_2}\setminus G_{h M}(u,\Omega))\cap
S_\phi(x_0,t_{x_0})\Big) \leq \mu\Big(S_\phi(x_0,t_{x_0})\setminus
G_{h M}(u,\Omega)\Big) < 4\epsilon \,\mu(S_\phi(x_0,t_{x_0})),
\]
a contradiction with \eqref{eq:criticaldensitynonhomoge}. So
\eqref{eq:sectioncontainedinomegaunionmaxfunctionbig} is proved and  we can  apply 
 the covering result \cite[Theorem~6.3.3]{G} to conclude that
\begin{align}\label{eq:powerdecaynonhomowithtildedelta}
&\mu(U_{\alpha_2}\setminus G_{h M}(u,\Omega))\\
&\leq 2 \sqrt{\epsilon}\left[
\mu(U_{\alpha_1}\setminus G_{h}(u,\Omega))
 +  \mu\big\{x\in U_{\alpha_0}:
\mathcal M_\mu ((f/\det D^2\phi)^n)(x)> (2\theta n h)^n\big\}\right],\nonumber
\end{align}
as long as $\alpha_2<\alpha_1<\alpha_0$ are such that $\eta_2<\eta_1$, and  $h\geq 1/\eta_0$ satisfy $1/\theta h \leq
\eta_2$.

For $k\in\N$, set
\[
a_k:= \mu(U_{\alpha_k}\setminus G_{M^k}(u,\Omega)) \quad \text{and}\quad b_k:=\mu\big\{x\in U_{\alpha_0}:
\mathcal M_\mu ((f/\det D^2\phi)^n)(x)> (2\theta n M^k)^n\big\}, 
\]
where $\alpha_k$ will be defined inductively in the sequel. First fix $\alpha_1$ so that 
$2\alpha_0-1<\alpha_1<\alpha_0$ and take
$\eta_1 := C_0 (\alpha_0-\alpha_1)^{\gamma}$,
where $C_0$ and $\gamma$  are the constants
in Lemma~\ref{containment2}. Let $h=M$, and set
$\alpha_2=\alpha_1 -\frac{1}{(C_0 \theta M)^{1/\gamma}}$. Then 
$\frac{1}{\theta h}
=\frac{1}{\theta M}=C_0 (\alpha_1-\alpha_2)^{\gamma}=: \eta_2$, and
so from Lemma~\ref{containment2} and \eqref{eq:powerdecaynonhomowithtildedelta}  we get
\begin{align*}
a_2\leq 2\sqrt{\epsilon}(a_1 + b_1).
\end{align*}
Next let $h=M^2$ and $\alpha_3 = \alpha_2 -\frac{1}{ (C_0 \theta M^2)^{1/\gamma}}$, so
$\frac{1}{\theta h}=C_0
(\alpha_2-\alpha_3)^{\gamma}=:\eta_3$. Then
\begin{align*}
a_3\leq 2\sqrt{\epsilon}(a_2 + b_2).
\end{align*}
Continuing in this way we let $h =M^k$ and $\alpha_{k+1}=
\alpha_k -\frac{1}{(C_0 \theta M^k)^{1/\gamma}}$.
Then $\frac{1}{\theta h} =C_0
(\alpha_k-\alpha_{k+1})^{\gamma}=:\eta_k$, and
$a_{k+1}\leq 2\sqrt{\epsilon}(a_k + b_k)$.
These imply that
\begin{align*}
a_{k+1}\leq (2\sqrt{\epsilon})^k a_1 + \sum_{i=1}^{k}{(2\sqrt{\epsilon})^{(k+1)-i} b_i}.
\end{align*}
On the other hand, $\alpha_{k+1}=\alpha_1- \sum_{j=1}^k \frac{1}{(C_0 \theta M^j)^{1/\gamma}} \geq \alpha_1 - \frac{1}{(C_0\theta)^{1/\gamma}} \frac{1}{M^{1/\gamma} -1}
\geq 2\alpha_0 -1$ by choosing  $\eta$ even smaller depending on $\alpha$ (recall that $M=2\theta/\eta$).
Therefore, we obtain
\begin{equation*}
\mu(U_{2\alpha_0 -1}\setminus G_{M^{k+1}}(u,\Omega)) \leq
\mu(U_{\alpha_{k+1}}\setminus G_{M^{k+1}}(u,\Omega))
\leq (2\sqrt{\epsilon})^k a_1 + \sum_{i=1}^{k}{(2\sqrt{\epsilon})^{(k+1)-i} b_i} 
\end{equation*}
for all $k=1,2,\dots$
Moreover, 
\begin{align*}
b_i \leq \dfrac{C(n,\lambda,\Lambda)}{M^{n i}}\, \int_U \Big|\frac{f}{\det
D^2 \phi}\Big|^n\, d\mu(x) \leq C(n,\lambda,\Lambda) M^{-n i}
\end{align*}
because  $\mathcal M_\mu$ is of weak type $1-1$ (see \cite[Theorem~2.9]{GT}).
Thus, by setting $m_0 :=\max\{2\sqrt{\epsilon}, M^{-n}\}$  we  then have
\begin{equation*}
\mu(U_{2\alpha_0 -1}\setminus G_{M^{k+1}}(u,\Omega)) \leq
m_0^k a_1 + C k m_0^{k+1}\leq C(\epsilon, n,\lambda,\Lambda) m_0^{k+1}(1 + k).
\end{equation*}
Writing
$m_1=\sqrt{m_0}$ and since $m_0<1$, we conclude  that $m_0^{k+1}
\,(1+k)\leq  C'(m_0)\, m_1^{k+1}$ and so
$\mu(U_{2\alpha_0 -1}\setminus G_{M^{k+1}}(u,\Omega)) \leq
 C\,
m_1^{k+1}$.
Now for any $\beta \geq M^2$, pick $k\in \N$ such that $M^{k+1}\leq \beta < M^{k+2}$, then
$k+1\leq \log_M \beta < k+2$ and
\begin{align*}\label{initial-est-LA}
\mu(U_{\alpha}\setminus G_\beta(u,\Omega))
&=
\mu(U_{2\alpha_0 -1}\setminus G_\beta(u,\Omega))\\
& \leq
\mu(U_{2\alpha_0 -1}\setminus G_{M^{k+1}}(u,\Omega)) \leq  C\,
m_1^{k+1} \leq \dfrac{C}{m_1} \, \beta^{\log_M m_1}.
\end{align*}
 \end{proof}
 
\section{$L^p$ estimates for second derivatives}\label{sec:L^p-estimate}
We established in Proposition~\ref{Initial-Estimate} that 
\[\mu (U_\alpha \setminus G_\beta(u,\Omega))\leq C \beta^{-\tau}
\]
 when $\lambda\leq \det D^2\phi\leq \Lambda$.
 This power decay estimate is very poor as $\tau>0$ is  small. However,  we will demonstrate in this section that
 $\tau$ can be taken to be any finite number provided that
$\det D^2\phi$ is sufficiently close to the constant $1$ in $L^\infty$ norm. In order to perform this acceleration process, the following approximation lemma is crucial. This lemma is a variant of  \cite[Lemma~4.1]{GN} and allows us to compare explicitly two solutions originating from two different linearized Monge-Amp\`ere equations.  We assume below that $\phi, w\in C(\overline U)$ are convex functions satisfying  $\frac{1}{2} \leq \det D^2 \phi\leq \frac{3}{2}$, $\det D^2 w =1$ in $U$ and $\phi = w =0$ on $\partial U$. Also the matrices of cofactors of $D^2\phi$ and $D^2w$ are denoted by $\Phi$ and   $\calW$ respectively. 

\begin{lemma}\label{explest} 
Let $U$ be a normalized convex domain and  $u\in W^{2,n}_{loc}(U)\cap C(\overline U)$ be a  solution of 
$\Phi_{i j} D_{i j}u = f$ in $U$
with $|u|\leq 1$ in  $U$. Assume $0<\alpha_1<1$ and $h\in W^{2,n}_{loc}(U_{\alpha_1})\cap C(\overline U_{\alpha_1})$  is a solution of
\begin{equation*}
\left\{\begin{array}{rl}
\calW_{i j} D_{i j}h  &= 0 \quad \mbox{ in }\quad U_{\alpha_1}\\
h  &=u\quad\mbox{ on } \quad\partial U_{\alpha_1}.
\end{array}\right.
\end{equation*}
Then there exists $\gamma \in (0,1)$ depending only on $n$ such that for any $0<\alpha_2<\alpha_1$, we have
\[
\|u-h\|_{ L^\infty(U_{\alpha_2})}
+\|f-\trace([\Phi-\calW ]D^2h)\|_{ L^n(U_{\alpha_2})}
\leq C(\alpha_1,\alpha_2,n) \left\{\| \Phi - \calW\|_{L^n(U_{\alpha_1})}^\gamma + \|f\|_{L^n(U)} \right\}
\]
provided that 
$\| \Phi - \calW\|_{L^n(U_{\alpha_1})} \leq (\alpha_1 -\alpha_2)^{\frac{2n}{1+ (n-1)\gamma}}$. 
\end{lemma}
\begin{proof}
Let $0<\alpha<\alpha_1$. We first claim that
\begin{equation}\label{boundary-distance}
\delta_1 := c_n (\alpha_1 - \alpha)^n\leq \dist(x ,\partial U_{\alpha_1}) \leq 
2n\min{\big\{1, \alpha^{-1}(\alpha_1 -\alpha)\big\}}=:\delta_2\quad \forall x\in \partial U_{\alpha}.
\end{equation}
To prove \eqref{boundary-distance}, let $x_0$ be the minimum point of $\phi$ in $U$. Then
$U_\alpha= S_\phi(x_0, -\alpha\phi(x_0))$, $U_{\alpha_1}= S_\phi(x_0, -\alpha_1\phi(x_0))$, and $C_1(n) \leq |\phi(x_0)|\leq C_2(n)$ by \cite[Proposition~3.2.3]{G}. For any $x\in \partial U_\alpha$, by applying Aleksandrov's estimate (see \cite[Theorem~1.4.2]{G}) to the function $\tilde\phi := \phi - (1-\alpha_1)\phi(x_0)$ we get $\dist(x,\partial U_{\alpha_1})^{1/n}\geq C_n |\tilde \phi(x)| = C_n (\alpha-\alpha_1)\phi(x_0)$ yielding the first inequality in \eqref{boundary-distance}. For the second inequality, let $x\in \partial U_\alpha$ and choose $y$ be such that $x= (1-\frac{\alpha}{\alpha_1}) x_0 + \frac{\alpha}{\alpha_1}y$. Then $y\not\in U_{\alpha_1}$ since whenever $y\in U$ we have $ (1-\alpha_1)\phi(x_0)\leq \phi(y)$ as
 $(1-\alpha)\phi(x_0)=\phi(x) \leq  (1-\frac{\alpha}{\alpha_1}) \phi(x_0) + \frac{\alpha}{\alpha_1}\phi(y)$ by the convexity of $\phi$. Therefore, we  infer that $\dist(x,\partial U_{\alpha_1})\leq |y-x|=(\frac{\alpha_1}{\alpha} -1) |x-x_0|\leq 2n (\frac{\alpha_1}{\alpha} -1)$ which gives the desired result.

By Caffarelli-Guti\'errez interior H\"older estimates (see \cite[estimate (2.2) and Corollary~2.6]{GN}) there exists $\beta\in (0, 1)$ depending only on $n$ such that
\begin{equation}\label{CG}
\|u\|_{C^\beta(\overline U_{\alpha_1})}\leq C(\alpha_1,n)\Big(1+ \|f\|_{L^n(U)} \Big).
\end{equation}
Next notice that  Pogorelov's estimates imply  that $\lambda(\alpha_1, n) I \leq \calW\leq \Lambda(\alpha_1, n) I$ in $U_{\alpha_1}$.  Therefore, by using standard boundary H\"older estimates for linear uniformly elliptic equations (see \cite[Corollary~9.29]{GiT} and \cite[Proposition~4.13]{CC}) and \eqref{CG}, we obtain
\begin{equation}\label{classical-boundary-est}
\|h\|_{C^{\beta/2}(\overline U_{\alpha_1})}\leq C'(\alpha_1,n)\|u\|_{C^{\beta}(\partial U_{\alpha_1})}
\leq
C(\alpha_1,n)\Big(1+ \|f\|_{L^n(U)} \Big).
\end{equation}

  Now for any $x\in \partial U_{\alpha}$, by \eqref{boundary-distance} we can take $y\in \partial U_{\alpha_1}$ such that $|x-y| \leq \delta_2$. Then since $u-h =0$ on $ \partial U_{\alpha_1}$, we get from \eqref{CG} and \eqref{classical-boundary-est} that
\begin{align*}
|(u-h)(x)| &= |(u-h)(x) - (u-h)(y)| \leq |u(x)-u(y)| + |h(x)-h(y)|\\
&\leq  C(\alpha_1,n) \, \delta_2^{\beta/2} \Big(1+ \|f\|_{L^n(U)} \Big).
\end{align*}
That is,
\begin{equation}\label{dest}
\|u-h\|_{L^\infty(\partial U_{\alpha}) } \leq C(\alpha_1,n) \, \delta_2^{\beta/2}\Big(1+ \|f\|_{L^n(U)} \Big).
\end{equation}
We claim that
\begin{equation}\label{2orderest}
\|D^2 h\|_{L^\infty(U_{\alpha}) } \leq C(\alpha_1,n) \,  \delta_1^{\frac{\beta}{2}-2} 
\Big(1+ \|f\|_{L^n(U)} \Big).
\end{equation}
Indeed, let $x_0\in U_{\alpha}$ be arbitrary and take $x_1\in \partial B_{\delta_1/2}(x_0)$. Since $B_{\delta_1/2}(x_0)\Subset U_{\alpha_1}$ by \eqref{boundary-distance} and
$\calW_{ij} D_{i j}(h- h(x_1))= \calW_{ij} D_{i j}h=0$ in $U_{\alpha_1}$, we can apply  interior $C^2$-estimates (see \cite[Theorem~2.7]{GN}) to $h-h(x_1)$ in $B_{\delta_1/2}(x_0)$  and obtain
\begin{equation*}
\|D^2 h(x_0)\| \leq C'(\alpha_1,n)\, \delta_1^{-2} \sup_{B_{\delta_1/2}(x_0)}{|h-h(x_1)|} \leq C(\alpha_1,n)\,  \delta_1^{-2} \delta_1^{\beta/2}
\Big(1+ \|f\|_{L^n(U)} \Big)
\end{equation*}
giving \eqref{2orderest}. 

Observe that $u-h\in W^{2,n}_{loc}(U)$ is a  solution of
\begin{align*}
\Phi_{ij} D_{i j}(u-h)=f - \Phi_{ij} D_{i j}h = f- [ \Phi_{i j} - \calW_{ij}] D_{i j}h =: F \quad \mbox{in}\quad  U_{\alpha_1}.
\end{align*}
Hence if we let $\eps:=\| \Phi - \calW\|_{L^n(U_{\alpha_1})}$, then it follows from the ABP estimate (see \cite[Theorem~2.4]{GN}), \eqref{dest} and \eqref{2orderest} that
\begin{align*}
&\|u-h\|_{L^\infty(U_{\alpha}) } +\|F\|_{L^n(U_{\alpha})}
\leq \|u-h\|_{L^\infty(\partial U_{\alpha}) }  + C_n \|F\|_{L^n(U_{\alpha_1})} \\
&\leq \|u-h\|_{L^\infty(\partial U_{\alpha}) } +C_n \|D^2 h\|_{L^\infty(U_{\alpha}) } \| \Phi - \calW\|_{L^n(U_{\alpha_1})}  + C_n \|f\|_{L^n(U_{\alpha_1})} \\ 
&\leq C(\alpha_1, n)  \Big[ \alpha^{-\beta/2} (\alpha_1 -\alpha)^{\beta/2}  + (\alpha_1 -\alpha)^{n(\frac{\beta}{2}-2)} \eps\Big]\Big(1+ \|f\|_{L^n(U)} \Big)    + C_n \|f\|_{L^n(U)}.
\end{align*}
By  taking $\alpha :=\alpha_1 - \eps^{\frac{2}{4n -(n-1)\beta}}$, this yields
\[\|u-h\|_{L^\infty(U_{\alpha}) } +\|F\|_{L^n(U_{\alpha})}
\leq C(\alpha_1,n) (\alpha^{-\beta/2} +1 ) \eps^{\gamma}\Big(1+ \|f\|_{L^n(U)} \Big) + C_n \|f\|_{L^n(U)}
\]
with $\gamma := \frac{\beta}{4n- (n-1) \beta}$.
From this we deduce the lemma as $\eps \leq (\alpha_1 -\alpha_2)^{\frac{2n}{1+ (n-1)\gamma}}$ by the assumption.
\end{proof}

\subsection{Improved density estimates}\label{sub:improveddensity}
In this subsection we will use 
Lemma~\ref{explest} to improve the power decay of $\mu (U_\alpha \setminus G_\beta(u,\Omega))$. To this end, the next  lemma plays an important role. 

\begin{lemma}\label{lm:acceleration}
Let $0<\epsilon<1/2$, $0<\alpha_0<1$, $U$ be a normalized convex domain  and $\Omega$ be a bounded convex set such that $U\subset \Omega$. Let $\phi\in  C^1(\Omega)\cap  W^{2,n}_{loc}(U)$ be  a convex function satisfying $1-\epsilon \leq \det D^2\phi \leq
1+\epsilon$ in $U$ and $\phi=0$ on $\partial U$.  Suppose $u\in C(\Omega)\cap W^{2,n}_{loc}(U)\cap C^1(U)$ is a solution of  
$\calL_\phi u=f$ in $U$ with  $|u|\leq 1$ in $U$ and $|u(x)|\leq C^* d(x, x_0)^2$ in $\Omega\setminus U$ for some $x_0\in U_{\alpha_0}$. 
Then for any $0< \alpha\leq \alpha_0$, there exist $C,\, \tau>0$ depending only on $\alpha$, $\alpha_0$ and $n$ such that
\begin{equation*}
|G_{N}(u,\Omega) \cap  U_{\alpha} |\geq \left\{1- C\big( N^{-\tau} \delta_0^\tau +\epsilon \big)\right\} \, | U_{\alpha}|
\end{equation*}
for any $N\geq N_0 = N_0(\alpha, \alpha_0, C^*, n)$ and provided that $\| \Phi - \calW\|_{L^n(U_{\frac{\alpha_0 +1}{2}})} \leq ((1 -\alpha_0)/4)^{\frac{2n}{1+ (n-1)\gamma}}$.  Here  $\calW,\, \gamma$ are from Lemma~\ref{explest} and
\[
\delta_0 := \Big(\fint_{U_{\frac{\alpha_0 +1}{2}}} \|\Phi -\calW\|^n ~dx\Big)^{\frac{\gamma}{ n}}
+ \Big(\fint_{U} |f|^n ~dx\Big)^{\frac{1}{ n}}. 
\]
\end{lemma}
\begin{proof}
Let $h\in W^{2,n}_{loc}(U_{\frac{\alpha_0 +1}{2}})\cap C(\overline{ U_{\frac{\alpha_0 +1}{2}}})$ be the solution of
\begin{equation*}
\left\{\begin{array}{rl}
\calW_{i j} D_{i j}h  &= 0 \qquad \mbox{ in } \quad U_{\frac{\alpha_0 +1}{2}}\\
h  &=u\ \ \ \ \ \ \ \ \mbox{ on } \quad\partial U_{\frac{\alpha_0 +1}{2}}.
\end{array}\right.
\end{equation*}
By the interior $C^{1,1}$ regularity of $h$ and Lemma~\ref{explest}, we have 
\begin{align}
&\|h\|_{C^{1,1}(U_{\frac{3\alpha_0 +1}{4}})}\leq c_e(\alpha_0,n) \|u\|_{L^\infty( U_{\frac{\alpha_0 +1}{2}})}\leq c_e(\alpha_0,n),\label{regularity-alpha}\\
&\|u-h\|_{L^\infty(U_{\frac{3\alpha_0 +1}{4}})} + \| f-\trace([\Phi-\calW] D^2h)\|_{L^n(U_{\frac{3\alpha_0 +1}{4}})}\leq C(\alpha_0,n)\,\delta_0 =:\delta_0'.\label{closeness-alpha}
\end{align}
We now consider $h|_{U_{\frac{3\alpha_0 +1}{4}}}$ and then extend $h$ outside $U_{\frac{3\alpha_0 +1}{4}}$ continuously such that
\begin{equation*}
\left\{\begin{array}{rl}
&h(x)  = u(x) \quad \forall  x\in \Omega\setminus U_{\frac{\alpha_0 +1}{2}},\\
&\|u-h\|_{L^\infty(\Omega)}  =\|u-h\|_{L^\infty(U_{(3\alpha_0 +1)/4})}.
\end{array}
\right.
\end{equation*}
Since by the maximum principle $\|h\|_{L^\infty(U_{\frac{3\alpha_0 +1}{4}})} \leq \|u\|_{L^\infty(U)}\leq 1$, we then obtain that 
\begin{equation}\label{h-above-below}
u(x) -2 \leq h(x) \leq u(x) +2 \quad \mbox{for all}\quad x\in \Omega.
\end{equation}
We claim that if $N\geq N_0 :=N_0(\alpha, \alpha_0, n)$, then
\begin{equation}\label{h-is-good}
U_{\alpha}\cap A_{\sigma(\alpha)}\subset G_N(h,\Omega)
\end{equation}
 where $\sigma(\alpha)>0$ is the constant given by \cite[Theorem~6.1.1]{G} and 
\[
A_{\sigma(\alpha)} := \left\{\tilde x\in U: \phi(x)\geq
\phi(\tilde x) + \nabla \phi(\tilde x)\cdot (x-\tilde x) + \frac{\sigma(\alpha)}{2}\, |x-\tilde x|^2, \quad \forall x\in U\right\}.
\]
Indeed, let $\bar x\in U_{\alpha} \cap A_{\sigma(\alpha)}\subset U_{\alpha_0}$.
By \eqref{regularity-alpha} we have  $ \left | h(x) -[ h (\bar x) +
\nabla  h (\bar x)\cdot ( x- \bar x)]\right|\leq c_e(\alpha_0,n) |x- \bar x|^2$
for all $x\in U_{\frac{3\alpha_0 +1}{4}}$, and since 
$\bar x\in  A_{\sigma(\alpha)}$ 
\begin{equation}\label{d-below}
d(x, \bar x)^2 = \phi(x) -[\phi(\bar x) +\nabla \phi(\bar x)\cdot( x-\bar x)]
\geq \frac{\sigma(\alpha)}{2} |x-\bar x|^2\quad\forall x\in \overline{U}.
\end{equation}
Therefore
\begin{equation}\label{h-step2}
\big| h(x) -[ h (\bar x) +\nabla h (\bar x)\cdot ( x- \bar x)] \big|\leq \frac{2 c_e(\alpha_0,n)}{\sigma(\alpha)} d(x, \bar x)^2\quad\forall x \in U_{\frac{3\alpha_0 +1}{4}}.
\end{equation}
We next show that by increasing the constant on the right hand side of \eqref{h-step2}, that the resulting inequality holds for all $x$ in $\Omega$. 
To see this, observe that
\begin{align}\label{d-via-d}
d(x, x_0)^2  &= d(x, \bar x)^2 
+ [\phi(\bar x) -\phi(x_0) - \nabla \phi(x_0)\cdot (\bar x -x_0) ]\\
& \qquad \qquad \, +  [\nabla \phi(\bar x) -\nabla \phi(x_0)]\cdot (x-\bar x) \nonumber\\
&\leq d(x, \bar x)^2 + C(\alpha_0, n)\,  (1 +  |x -\bar x| ) \qquad  \mbox{for all }\quad x\in\Omega.\nonumber
\end{align}
Also there exists $c(\alpha,n)>0$ such that 
\begin{equation}\label{d-below-linear}
 d(x, \bar x)^2 \geq c(\alpha,n) \, |x-\bar x| \quad \forall x\in \overline{\Omega}\setminus U.
\end{equation}
Notice that $\dist (U_\alpha,\partial U)\geq c_n (1-\alpha)^n$ by the Aleksandrov estimate \cite[Theorem~1.4.2]{G} and \cite[Proposition~3.2.3]{G}.  Thus it follows from   \eqref{d-below}  and the fact $\bar x\in U_{\alpha}$ that there is $c=c(\alpha,n)>0$ so that \eqref{d-below-linear} holds for all $x\in \partial U$. Now for $x\in \overline{\Omega}\setminus \overline{U}$ we can choose $\hat x \in \partial U$ and $\lambda\in (0,1)$ satisfying $\hat x = \lambda x + (1-\lambda) \bar x$. Then since $d(\hat x, \bar x)^2 \geq c |\hat x -\bar x|$ and the function $z\mapsto d(z,\bar x)^2$ is convex, we obtain 
\[
\lambda d(x, \bar x)^2 + (1-\lambda) d(\bar x, \bar x)^2
\geq c |\lambda x + (1-\lambda) \bar x -\bar x| = c\lambda |x-\bar x|
\]
which gives $d(x, \bar x)^2 \geq c |x-\bar x|$ and hence \eqref{d-below-linear} is proved.

We are ready to show that \eqref{h-step2} holds for all $x\in\Omega$ but with a bigger constant on the right hand side. Let $x\in \Omega \setminus U_{\frac{3\alpha_0 +1}{4}}$ and consider the following cases:

{\bf Case 1:} $x\in U$. Then by using \eqref{regularity-alpha}, \eqref{h-above-below}  and the assumption $|u|\leq 1$ in $U$, we have
\begin{align*}
\left | h(x) -[ h (\bar x) +
\nabla  h (\bar x) \cdot (x- \bar x) ]\right|
&\leq |h(x) - h(\bar x)| +C(\alpha_0,n)
\leq |u(x) - u(\bar x)| +C(\alpha_0,n)\\ &\leq C(\alpha_0,n)
\leq C_1(\alpha_0,n) \, d(x, \bar x)^2
\end{align*}
where in the last inequality we have used the fact that since $\bar x\in U_{\alpha}\subset U_{\alpha_0}$ there exists $\eta( \alpha_0)>0$ such that $S_\phi(\bar x, \eta(\alpha_0))\subset U_{\frac{3\alpha_0 +1}{4}}$ (see Lemma~\ref{containment2}).

{\bf Case 2:} $x\in \Omega\setminus  U$. Then
$d(x,\bar x)^2\geq \eta_0$ since $S_\phi(\bar x, \eta_0)\Subset U$ by Lemma~\ref{containment2}. This together with the assumptions, 
\eqref{regularity-alpha}, \eqref{h-above-below}, \eqref{d-via-d} and \eqref{d-below-linear} gives 
\begin{align*}
&\left | h(x) -[ h (\bar x) +
\nabla  h (\bar x) \cdot (x- \bar x) ]\right|
\leq |h(x) - h(\bar x)| + C(\alpha_0,n) |x-\bar x|\\
&\leq |u(x)| + C(\alpha_0,n)\, ( |x-\bar x|+ 1)
\leq  C^*\,d(x, x_0)^2 +C(\alpha_0,n)\, ( |x-\bar x|+ 1) \\
&\leq   C^*\,d(x, \bar x)^2 +C(\alpha_0,n)\, ( |x-\bar x|+ 1)
\leq C_2(\alpha,\alpha_0, C^*, n) \, d(x, \bar x)^2.
\end{align*}
Therefore if we choose 
\[N_0 := \max{\Big\{\frac{ 2 c_e(\alpha_0,n)}{\sigma(\alpha)}, C_1(\alpha_0,n), C_2(\alpha,\alpha_0, C^*,n) \Big\}},
\]
then it follows from the above considerations and \eqref{h-step2} that 
\[
\left | h(x) -[ h (\bar x) +
\nabla  h (\bar x) \cdot (x- \bar x) ]\right|\leq N_0 \, d(x, \bar x)^2
\quad\mbox{for all}\quad x \in  \Omega.
\]
This means $\bar x\in G_{N_0}(h,\Omega)\subset G_{N}(h,\Omega)$ for all $N\geq N_0$. Thus claim \eqref{h-is-good} is proved.

Next let
\[
u'(x) := \frac{(u-h)(x)}{\delta_0'}, \quad \mbox{for}\quad x\in\Omega.
\]
We infer from \eqref{closeness-alpha} and the way $h$ was initially defined and extended that
\begin{align*}
& \|u'\|_{L^\infty(\Omega)} =\frac{1}{\delta_0'} \|u-h\|_{L^\infty(U_{\frac{3\alpha_0 +1}{4}})}\leq 1,\\
& \calL_\phi u' =\frac{1}{\delta_0'} [\calL_\phi u -\calL_\phi h ]
 =\frac{1}{\delta_0'} \Big[f-\trace([\Phi -\calW] D^2 h) \Big]
=: f'(x) 
\qquad \mbox{ in}\quad U_{\frac{3\alpha_0 +1}{4}}.
\end{align*}
Notice that $\|f'\|_{L^n(U_{\frac{3\alpha_0 +1}{4}})}\leq 1$ by \eqref{closeness-alpha}. In order to apply Proposition~\ref{Initial-Estimate}, let $T$ be an invertible affine map normalizing $U_{\frac{3\alpha_0 +1}{4}}$. We have
$C(n)\leq |\det T|\leq C'(n)$ because
 $|\det T|^{-2/n}\approx |U_{\frac{3\alpha_0 +1}{4}}|^{2/n}\approx \frac{3\alpha_0 +1}{4} |\min_{U}{\phi}|\approx 1$. Set $\tilde U := T(U_{\frac{3\alpha_0 +1}{4}})$,
$\tilde\Omega :=T(\Omega)$ and define 
\begin{align*}
\tilde \phi(y)=
|\det T|^{2/n}\left[
\phi(T^{-1}y) - \big(1-\frac{3\alpha_0 +1}{4}\big)\min_{U}{\phi}\right],\quad 
 \tilde u(y)&=  u'(T^{-1}y)\quad \mbox{for } y\in \tilde\Omega.
 \end{align*}
 Then $1-\epsilon \leq \det D^2\tilde\phi = \det  D^2\phi(T^{-1} y) \leq  1+\epsilon$ in $\tilde U$ and $\tilde \phi=0$ on $\partial \tilde U$.
 Moreover since $\tilde \Phi(y)=|\det T|^{-2/n}
\det D^2\phi(T^{-1}y)  ~T [D^2\phi(T^{-1}y)]^{-1} T^t$, we obtain
\begin{align*}
&\trace\big(\tilde\Phi(y) D^2\tilde u(y)\big)= 
|\det T|^{\frac{-2}{n}}\trace\big(\Phi(T^{-1}y) D^2 u'(T^{-1}y)\big) = |\det T|^{\frac{-2}{n}} f'(T^{-1}y) =: \tilde f(y)
\end{align*}
in $\tilde U$. Thus as 
$\tilde\phi\in  C^1(\tilde\Omega)\cap  W^{2,n}(\tilde U)$,
$\tilde u \in C(\tilde\Omega)\cap W^{2,n}(\tilde U)\cap C^1(\tilde U)$, $\|\tilde u \|_{L^\infty(\tilde \Omega)} =\|u' \|_{L^\infty(\Omega)}\leq 1$ and
$\|\tilde f\|_{L^n(\tilde U)}= |\det T|^{-1/n} \|f'\|_{L^n(U_{\frac{3\alpha_0 +1}{4}})}\leq |\det T|^{-1/n} \leq C_n$,
we can  apply Proposition~\ref{Initial-Estimate} to get
\begin{equation}\label{normalized-est}
| \tilde U_{\frac{4 \alpha}{3\alpha_0 +1}}\setminus G_{\frac{N}{\delta_0'} |\det T|^{\frac{-2}{n}}}(\tilde u,\tilde \Omega)|
\leq C(\alpha,\alpha_0,n) \left(\frac{\delta_0}{N} \right)^{\tau}, 
\end{equation}
where $\tau>0$ depends only on $\alpha$, $\alpha_0$ and $n$. Since  $\tilde U_{\frac{4 \alpha}{3\alpha_0 +1}} = T(U_\alpha)$ 
and  
\begin{equation*}
\tilde d(T x,T \bar x)^2
= \tilde\phi(Tx) -\tilde\phi(T\bar x) -\langle \nabla \tilde\phi (T\bar x), Tx - T\bar x \rangle
 = |\det T|^{\frac{2}{n}} \, d(x, \bar x)^2\quad \forall x,\,\bar x\in\Omega,
\end{equation*}
we have
\begin{align}\label{G-transformation}
&G_{\beta}(\tilde u, \tilde \Omega) \cap \tilde U_{\frac{4 \alpha}{3\alpha_0 +1}}\\
&= T \left\{  \bar x\in U_\alpha: 
\left |\tilde u(y) -[\tilde u (T \bar x) +\langle
\nabla \tilde u (T \bar x), y- T \bar x \rangle]\right|\leq \beta \tilde d(y,T \bar x)^2 \quad \forall y\in T(\Omega)\right\}\nonumber\\
&= T \left\{  \bar x\in U_\alpha: 
\left |u'(x) -[u' (\bar x) +\langle
\nabla u' ( \bar x), x-  \bar x \rangle]\right|\leq \beta |\det T|^{\frac{2}{n}}  d( x, \bar x)^2 \quad \forall x\in \Omega\right\}\nonumber\\
&= T \Big( G_{\beta |\det T|^{\frac{2}{n}}}(u',\Omega) \cap U_\alpha\Big).\nonumber
\end{align}
It follows from \eqref{normalized-est}, \eqref{G-transformation} and the fact $T(A)\setminus T(B) =T(A\setminus B)$ that
\begin{equation*}
| U_{\alpha}\setminus G_{\frac{N}{\delta_0'}}(u',\Omega)|
\leq C(\alpha,\alpha_0,n) \left(\frac{\delta_0}{N} \right)^{\tau}.
\end{equation*}
As 
$ G_{\frac{N}{\delta_0'}}(u',\Omega) =  G_{N}(u-h,\Omega)$ and $|U_\alpha|\geq c_n \alpha^{n/2}$, we then conclude 
\begin{align*}
| U_{\alpha}| - |G_{N}(u -h,\Omega) \cap  U_{\alpha}|=
| U_{\alpha}\setminus G_{N}(u -h,\Omega)| \leq 
 C \left(\frac{\delta_0}{N} \right)^{\tau}  ~ |U_{\alpha}|
\end{align*}
yielding
\begin{align*}
 \left\{ 1- 
 C  \big(\frac{\delta_0}{N} \big)^{\tau}  \right\}~ |U_{\alpha}|
 &\leq |G_{N}(u -h,\Omega) \cap  U_{\alpha}|\\
 & \leq |G_{N}(u -h,\Omega) \cap  U_{\alpha}\cap A_{\sigma(\alpha)}| +\left|U_{\alpha}\setminus A_{\sigma(\alpha)}\right|\\
 & \leq |G_{N}(u -h,\Omega) \cap  U_{\alpha}\cap A_{\sigma(\alpha)}| + C(\alpha,n) \, \epsilon  \, |U_{\alpha}|,
\end{align*}
where the last inequality is from \cite[Theorem~6.1.1]{G}. Consequently,
\begin{equation}\label{Est-12}
 |G_{N}(u -h,\Omega) \cap  U_{\alpha}\cap A_{\sigma(\alpha)}| \geq \left\{ 1- 
 C \Big[ \big(\frac{\delta_0}{N} \big)^{\tau} + \epsilon\Big] \right\}~ |U_{\alpha}|.
\end{equation}
We claim that
\begin{equation}\label{h-via-u}
G_{N}(u -h,\Omega) \cap  U_{\alpha}\cap A_{\sigma(\alpha)} \subset 
G_{2 N}(u,\Omega ) \cap  U_{\alpha}
\end{equation}
which together with \eqref{Est-12}  gives the conclusion of the lemma. To prove the claim, let 
 $\bar x \in G_{N}(u -h,\Omega) \cap U_{\alpha}\cap A_{\sigma(\alpha)}$. Then
$\bar x \in G_{N}(u -h,\Omega) \cap  G_N(h,\Omega)$ by \eqref{h-is-good}. 
Therefore \eqref{h-via-u} holds because
\begin{align*}
&\left | u(x) - [u(\bar x) +\langle \nabla u(\bar x), x-\bar x\rangle]\right| \\
&\leq \left | (u-h)(x) - [(u-h)(\bar x) +\langle \nabla (u-h)(\bar x), x-\bar x\rangle] \right| + \left| h(x) - [h(\bar x) +\langle\nabla h(\bar x), x-\bar x\rangle]\right| \\
&\leq 2N d(x,\bar x)^2 \qquad \mbox{for all}\quad x\in \Omega.
\end{align*}
This completes the proof of the lemma.
\end{proof}

By using Lemma~\ref{lm:acceleration} and a localization process, we shall prove the following.

\begin{lemma}\label{lm:improved-density-I}
Let $0<\epsilon_0<1$, $0<\alpha_0<1$, $\Omega$ be a normalized convex domain and 
$u\in W^{2,n}_{loc}(\Omega)\cap C^1(\Omega)$ be a solution of
$\calL_\phi u=f$ in $\Omega$ with $\|u\|_{L^\infty(\Omega)}\leq 1$, where  
 $\phi\in C(\overline\Omega)$ is a convex function satisfying $\phi=0$ on $\partial \Omega$. 
There exists $\epsilon>0$ depending only on $\epsilon_0$, $\alpha_0$ and $n$ such that 
if $1-\epsilon \leq \det D^2\phi \leq
1+\epsilon$ in $\Omega$, then 
for any $S_\phi(x_0, \frac{t_0}{\alpha_0}) \subset \Omega_{\frac{\alpha_0 +1}{2}}$  we have
\begin{equation}\label{localized-est}
 \left| G_{\frac{N}{t_0}}(u,\Omega) \cap S_\phi(x_0, t_0)\right|
 \geq \left\{1-\epsilon_0 -C \big(\frac{t_0}{N}\big)^\tau  \Big(\fint_{S_\phi(x_0, \frac{t_0}{\alpha_0})} | f|^n ~dx\Big)^{\frac{\tau}{n}} \right\}~ \left|S_\phi(x_0, t_0)\right|
 \end{equation}
for every  $N\geq N_0$. Here $C$, $\tau$ and $N_0$ are positive constants depending only on $\alpha_0$ and $n$.
\end{lemma}
\begin{proof}
Observe that in fact $\phi\in C^1(\Omega)$. As $\eps$ will be chosen small, we  also have $\phi\in W^{2,n}(\Omega_{\frac{\alpha_0 +1}{2}})$ by Caffarelli $W^{2,p}$ estimates
(see \cite[Theorem~1]{C3} and  \cite[Theorem~6.4.1]{G}).
 Let $T$ be an affine map normalizing $S_\phi(x_0, \frac{t_0}{\alpha_0})$ and let $U:=T\big(S_\phi(x_0, \frac{t_0}{\alpha_0})\big)$. For each $y\in  T(\Omega)$, define 
\begin{align*}
\tilde \phi(y)=
|\det T|^{2/n}\left[
\phi(T^{-1}y) -\phi(x_0) - \nabla \phi(x_0)\cdot (T^{-1} y-x_0) - \frac{ t_0}{\alpha_0}\right] \mbox{ and }
 \tilde u(y)&=  u(T^{-1}y).
 \end{align*}
 Then $1-\epsilon \leq \det D^2\tilde\phi = \det  D^2\phi(T^{-1} y) \leq  1+\epsilon$ in $U$ and $\tilde \phi=0$ on $\partial U$.
 Moreover 
\begin{align*}
&\trace\big(\tilde\Phi(y) D^2\tilde u(y)\big)= 
|\det T|^{\frac{-2}{n}}\trace\big(\Phi(T^{-1}y) D^2 u(T^{-1}y)\big) = |\det T|^{\frac{-2}{n}} f(T^{-1}y) =: \tilde f(y).
\end{align*}
Thus as $\|\tilde u \|_{L^\infty(T(\Omega))} =\|u \|_{L^\infty(\Omega)}\leq 1$,
we obtain from Lemma~\ref{lm:acceleration} with $\alpha := \alpha_0$ that
\begin{equation*}
|G_{N}(\tilde u, T(\Omega)) \cap  U_{\alpha_0} |\geq \left\{1- C\big(N^{-\tau} \delta_0^{\tau} +\epsilon\big)\right\} \, | U_{\alpha_0}|
\end{equation*}
for any  $N\geq N_0 = N_0(\alpha_0, n)$ and provided that $\| \tilde \Phi - \tilde \calW\|_{L^n(U_{\frac{\alpha_0 +1}{2}})} \leq ((1 -\alpha_0)/4)^{\frac{2n}{1+ (n-1)\gamma}}$,
where
\begin{equation}\label{delta-zero}
\delta_0 := \Big(\fint_{U_{\frac{\alpha_0 +1}{2}}} \|\tilde \Phi -\tilde \calW\|^n ~dy\Big)^{\frac{\gamma}{ n}}
+ \Big(\fint_{U} |\tilde f|^n ~dy\Big)^{\frac{1}{ n}}, 
\end{equation}
$\gamma$ is given by Lemma~\ref{explest} and
 $\tilde \calW$ is the cofactor matrix of $D^2 \tilde w$ with  $\tilde w$ is the convex function satisfying  $\det D^2 \tilde w =1$ in $U$ and $\tilde w=0$ on $\partial U$. This together with Lemma~\ref{convergence-of-cofactors} below implies that there exists  $\epsilon>0$ sufficiently small depending only on $\epsilon_0$, $\alpha_0$ and  $n$ such that
\begin{align*}\label{G-tilde-u}
 |G_{N}(\tilde u, T(\Omega)) \cap  U_{\alpha_0}|
 &\geq \left\{1-\epsilon_0 -C N^{-\tau}\Big(\fint_{U} |\tilde f|^n ~dy\Big)^{\frac{\tau}{n}}\right\}~ |U_{\alpha_0}|\\
 &=\left\{1-\epsilon_0 -C N^{-\tau}|\det T|^{\frac{-2\tau}{n}}\Big(\fint_{S_\phi(x_0,  \frac{t_0}{\alpha_0})} | f|^n ~dx\Big)^{\frac{\tau}{n}} \right\}~ |U_{\alpha_0}|\nonumber\\
 &\geq \left\{1-\epsilon_0 -C \big(\frac{t_0}{N}\big)^\tau \Big(\fint_{S_\phi(x_0,   \frac{t_0}{\alpha_0})} | f|^n ~dx\Big)^{\frac{\tau}{n}} \right\}~ |U_{\alpha_0}|\nonumber.
 \end{align*}
But since $U_{\alpha_0} =T(S_\phi(x_0, t_0))$,
the same calculations leading to \eqref{G-transformation} 
yield
\begin{align*}
&G_{N}(\tilde u, T(\Omega)) \cap U_{\alpha_0}
&= T \Big( G_{N |\det T|^{\frac{2}{n}}}(u,\Omega) \cap S_\phi(x_0,  t_0)\Big)\approx T \Big( G_{\frac{N}{t_0}}(u, \Omega) \cap S_\phi(x_0,  t_0)\Big).
\end{align*}
Therefore  we obtain
\begin{align*}
 \left| T \Big( G_{\frac{N}{t_0}}(u,\Omega) \cap S_\phi(x_0,  t_0)\Big)\right|
 &\geq \left\{1-\epsilon_0 -C \big(\frac{t_0}{N}\big)^\tau \Big(\fint_{S_\phi(x_0, \frac{t_0}{\alpha_0})} | f|^n ~dx\Big)^{\frac{\tau}{n}} \right\}~ \left|T (S_\phi(x_0, t_0))\right|
 \end{align*}
giving \eqref{localized-est} for any $N\geq N_0$.  
\end{proof}

In the above proof, we have used the following lemma which is a strengthen version of Lemma~3.5 in \cite{GN}. This result is proved by using a compactness argument and \cite[Lemma~3.5]{GN}.  
 
\begin{lemma}\label{convergence-of-cofactors}
Given any $0<\epsilon_0 <1$ and $0<\alpha<1$, there exists $\epsilon>0$ depending only on $\epsilon_0$, $\alpha$ and $n$ such that if  $\Omega\subset \R^n$ is a normalized convex domain and $\phi,\,w\in C(\overline \Omega)$ are convex functions satisfying
\begin{equation*}
\left\{\begin{array}{rl}
1-\epsilon\!\!\leq\! \det D^2 \phi \!\!\!\!\!&\leq 1+\epsilon\quad \mbox{in}\quad \Omega\\
\!\!\phi\!\!\!\!\!&=0 \ \ \quad \quad\mbox{on}\quad\partial\Omega
\end{array}\right.
\qquad \mbox{and}\qquad \left\{\begin{array}{rl}
 \det D^2 w \!\!\!\!\!&= 1 \quad
 \mbox{in}\quad \Omega\\
\!\!w\!\!\!\!\!&=0 \quad
\mbox{on}\quad\partial\Omega,
\end{array}\right.
\end{equation*}  
then 
\[
\|\Phi-\calW\|_{L^n(\Omega_\alpha)}\leq \epsilon_0,
\]
where $\quad\Omega_{\alpha} :=\{x\in \Omega: \phi(x)<(1-\alpha
)\,\min_{\Omega}\phi \}$.
\end{lemma}
\begin{proof}
Suppose by contradiction that it is not true. Then there exist $\eps_0, \alpha\in (0,1)$, $n\in \N$, a sequence of normalized convex domains $\Omega^k$ and  sequences of convex functions $\phi_k, w_k\in C(\overline{\Omega^k})$ with 
\begin{equation*}
\left\{\begin{array}{rl}
1-\frac{1}{k}\!\!\leq\! \det D^2 \phi_k \!\!\!\!\!&\leq 1+\frac{1}{k}\ \ \mbox{ in}\quad \Omega^k\\
\!\!\phi_k\!\!\!\!\!&=0\ \ \ \ \ \ \ \ \mbox{ on}\quad \partial\Omega^k
\end{array}\right.
\qquad \mbox{and}\qquad 
\left\{\begin{array}{rl}
 \det D^2 w_k\!\!\!\!\!&= 1\ \ \mbox{ in}\quad \Omega^k\\
\!\!w_k\!\!\!\!\!&=0\ \ \mbox{ on}\quad \partial\Omega^k
\end{array}\right.
\end{equation*}  
such that
\begin{equation}\label{contraass}
\|\Phi_k-\calW_k\|_{L^n(\Omega^k_\alpha)}\geq \epsilon_0\quad\mbox{ for all}\quad k.
\end{equation}

By Blaschke selection theorem, there is a subsequence of $\Omega^k$, still denoted by $\Omega^k$, such that $\Omega^k$ converges in the Hausdorff metric to a normalized convex domain $\Omega$. Also by 
\cite[Lemma 5.3.1]{G}
we have up to a subsequence $\phi_k\to\phi$ and $w_k\to w$ uniformly on compact subsets of $\Omega$, where $\phi, w\in C(\bar \Omega)$ are both convex solutions to the equation
\begin{equation*}
\left\{\begin{array}{rl}
\!\!\det D^2 w \!\!\!\!\!&= 1\ \ \mbox{ in}\quad \Omega,\\
\!\!w\!\!\!\!\!&=0\ \ \mbox{ on}\quad \partial\Omega.
\end{array}\right.
\end{equation*}
Thus $\phi\equiv w$ by the uniqueness of convex solutions to the Monge-Amp\`ere equation.

 Next observe that the Aleksandrov estimate \cite[Theorem~1.4.2]{G} and \cite[Proposition~3.2.3]{G} yield
 \begin{equation}\label{away-boundary}
 \dist (\Omega^k_\alpha,\partial \Omega^k)\geq c_n (1-\alpha)^n =: \tau\quad \forall k.
 \end{equation}
For $E\subset \R^n$, let $E(r) := \{x\in E: \dist(x,\partial E)>r \}$ and $\delta_r (E) :=\{x\in\R^n: \dist(x, E)<r\}$. We then claim that
 \begin{equation}\label{all-are-close}
 \Omega^k_\alpha\subset \Omega(\tau/2)\subset \Omega(\tau/4)\subset \Omega^k\quad \mbox{for all $k$ sufficiently large}.
 \end{equation}
Indeed, it follows from \eqref{away-boundary} that $\Omega^k_\alpha\subset \Omega^k(\tau)$. Moreover since $\Omega^k\to \Omega$ in the Hausdorff metric, we have $\Omega^k \subset \delta_{\frac{\tau}{2}}(\Omega)$ for all $k$ large (see \cite{Sc} for the definition of the Hausdorff distance). Therefore, $\Omega^k_\alpha\subset \delta_{\frac{\tau}{2}}(\Omega)(\tau)=\Omega(\tau/2)$ giving the first inclusion in \eqref{all-are-close}. We also infer from the Hausdorff convergence of $\Omega^k$ to $\Omega$ that $\Omega\subset \delta_{\frac{\tau}{4}}(\Omega^k)$  for all $k$ large. This implies $\Omega(\tau/4)\subset \delta_{\frac{\tau}{4}}(\Omega^k)(\tau/4)=\Omega^k$ and the last inclusion in \eqref{all-are-close} is proved.
   
By \eqref{all-are-close} and \cite[Lemma~3.5]{GN} we get
$\Phi_k \longrightarrow \Phi$ in $L^n(\Omega(\tau/2))$ and $\calW_k  \longrightarrow  \calW$ in $L^n(\Omega(\tau/2))$,
where $\Phi$ is the cofactor matrix of $D^2 \phi$ and $\calW$ is the cofactor matrix of $D^2 w$. Since 
 $\Phi \equiv \calW$, this yields $ \Phi_k - \calW_k \longrightarrow 0$ in $L^n(\Omega(\tau/2))$. Combining this with the first inclusion in \eqref{all-are-close} we obtain
 \begin{equation*}\label{cofconv}
\lim_{k\to\infty}\| \Phi_k - \calW_k\|_{L^n(\Omega^k_\alpha)}  =0,
\end{equation*}
which is a contradiction with \eqref{contraass} and the proof is complete.
\end{proof}

In the next lemma, we no longer require $\|u\|_{L^\infty(\Omega)}\leq 1$ as in Lemma~\ref{lm:improved-density-I}.

\begin{lemma}\label{lm:furthercriticaldensitytwo}
Let $0<\epsilon_0<1$, $0<\alpha_0<1$, $\Omega$ be a normalized convex domain
and $u\in W^{2,n}_{loc}(\Omega)\cap C^1(\Omega)$ be a solution of $\calL_\phi u=f$ in $\Omega$,
where $\phi\in C(\overline\Omega)$ is a convex function satisfying 
$\phi=0$ on $\partial \Omega$.
There exists $\epsilon>0$ depending only on  $\epsilon_0$, $\alpha_0$ and $n$  such that if $1-\epsilon \leq \det D^2\phi \leq 1+\epsilon$ in $\Omega$, then for 
any $S_\phi(x_0,t_0)\subset \Omega_{\alpha_0}$ with   $t_0\leq \eta(\alpha_0)$ and $S_\phi(x_0,t_0)\cap G_\gamma(u,\Omega)\neq \emptyset$ we have
\begin{align*}
 \left| G_{N \gamma }(u,\Omega) \cap S_\phi(x_0,  t_0)\right|
 &\geq \left\{1-\epsilon_0 -C (N\gamma)^{-\tau} \Big(\fint_{S_\phi(x_0, \frac{t_0}{\alpha_0})} | f|^n ~dx\Big)^{\frac{\tau}{n}} \right\}\left|S_\phi(x_0, t_0)\right|
 \end{align*}
for all $N\geq N_0$.
Here $\eta(\alpha_0)$, $C$, $\tau$ and $N_0$ are  constants depending only on $\alpha_0$ and $n$.
\end{lemma}
\begin{proof}
Let $\theta>1$ be the engulfing constant corresponding to  $1/2 \leq \det D^2\phi \leq 3/2$ in $\Omega$ and so $\theta$ depends only on the dimension $n$. By Lemma \ref{containment2}, there exists $\eta(\alpha_0) =\eta(\alpha_0,n)>0$ such that $S_\phi(x, \frac{\theta t}{\alpha_0})\subset \Omega_{\frac{\alpha_0 +1}{2}}$ for all $x\in\Omega_{\alpha_0}$ and $t\leq \eta(\alpha_0)$. We note that $\phi\in C^1(\Omega)\cap W^{2,n}(\Omega_{\frac{\alpha_0 +1}{2}})$ as explained in the proof of Lemma~\ref{lm:improved-density-I}.

Let $T$ normalize
$S_\phi(x_0, \frac{t_0}{\alpha_0})$ and  $U:= T\big(S_\phi(x_0, \frac{t_0}{\alpha_0})\big)$. For each $y\in T(\Omega)$, set
\begin{align*}
\tilde \phi(y)=|\det T|^{\frac{2}{n}} \Big[\phi(T^{-1}y) - \phi(x_0) -\langle \nabla \phi( x_0), T^{-1} y - x_0\rangle - \frac{ t_0}{\alpha_0} \Big]
 \text{ and }  \tilde u(y)= \dfrac{1}{2\theta t_0}\, u(T^{-1}y).
\end{align*}
We have $U $ is
normalized, $1-\epsilon
\leq \det D^2 \tilde \phi \leq 1 +\epsilon$ in $U$ and  $\tilde \phi=0$ on $\partial U$.
Let $\bar x\in S_\phi(x_0,t_0)\cap G_\gamma(u,\Omega)$ and $\bar y=T\bar x$. Then
\[
-\gamma \, d(x,\bar x)^2 \leq u(x)-u(\bar x) - \nabla u(\bar
x)\cdot (x-\bar x) \leq \gamma \, d(x,\bar x)^2,\quad \forall x\in \Omega.
\]
Hence by changing variables we get
\begin{equation}\label{eq:utildeboundedbydistance}
-\gamma \, \dfrac{d(T^{-1}y,T^{-1}\bar y)^2}{2\theta t_0} \leq \tilde
u(y)- \tilde u(\bar y) - \nabla \tilde u(\bar y)\cdot (y-\bar y)
\leq \gamma \, \dfrac{d(T^{-1}y,T^{-1}\bar y)^2}{2\theta t_0},~ \forall y\in T(\Omega).
\end{equation}

Since $\bar x\in S_\phi(x_0, t_0) \subset S_\phi(x_0,t_0/\alpha_0)$, we have $S_\phi(x_0, t_0/\alpha_0)\subset
S_\phi(\bar x, \theta t_0/\alpha_0)$ by the engulfing property. It follows that
$d(x,\bar x)^2\leq \theta t_0/\alpha_0$
for
$x\in S_\phi(x_0, t_0/\alpha_0)$ yielding $d(T^{-1}y,T^{-1}\bar y)^2\leq \theta t_0/\alpha_0$ for all
$y\in U$.  Consequently,
\[
-\frac{\gamma}{2\alpha_0} \leq \tilde u(y)- \tilde u(\bar y) - \nabla \tilde
u(\bar y)\cdot (y-\bar y) \leq \frac{\gamma}{2\alpha_0}\quad \forall y\in U.
\]
Let
$
v(y) := \dfrac{2 \alpha_0}{\gamma}\left[\tilde u(y)- \tilde u(\bar y) - \nabla \tilde u(\bar
y)\cdot (y-\bar y)\right]$, for $y\in T(\Omega)$.
Then $|v| \leq 1$ in $U$ and
by \eqref{eq:utildeboundedbydistance} we also have
\[
|v(y)|\leq \frac{\alpha_0}{\theta t_0}
d(T^{-1} y, T^{-1} \bar y)^2
\leq \frac{C_n\alpha_0}{\theta}
\tilde d( y, \bar y)^2\quad \forall y\in T(\Omega),
\]
where $\tilde d(y,\bar y)^2
:= \tilde\phi(y) -\tilde\phi(\bar y) -\langle \nabla \tilde\phi (\bar y), y - \bar y\rangle =|\det T|^{2/n}d(T^{-1}y,T^{-1}\bar y)^2$.
Moreover
\begin{align*}
\trace(\tilde \Phi D^2 v)
&= \frac{\alpha_0 |\det T|^{\frac{-2}{n}}}{ \theta \gamma t_0}\trace\Big(  \Phi(T^{-1} y) D^2u(T^{-1}y)  \Big)= \frac{\alpha_0}{\theta \gamma t_0 |\det T|^{\frac{2}{n}}} f(T^{-1} y)=: \tilde f(y).
\end{align*}
Notice that $\bar y\in U_{\alpha_0}= T(S_\phi(x_0,  t_0))$ because $\bar x\in S_\phi(x_0,  t_0)$.
Thus we obtain from Lemma~\ref{lm:acceleration} with $\alpha := \alpha_0$ that
\begin{equation*}
|G_{N}(v, T(\Omega)) \cap  U_{\alpha_0} |\geq \left\{1- C\big( N^{-\tau}\delta_0^{\tau} +\epsilon\big)\right\} \, | U_{\alpha_0}|
\end{equation*}
for any  $N\geq N_0 = N_0(\alpha_0,n)$ and provided that $\| \tilde \Phi - \tilde \calW\|_{L^n(U_{\frac{\alpha_0 +1}{2}})} \leq ((1 -\alpha_0)/4)^{\frac{2n}{1+ (n-1)\gamma}}$,
where
$\delta_0$ and $\tilde \calW$ are as in \eqref{delta-zero}.
This together with Lemma~\ref{convergence-of-cofactors} implies that there exists  $\epsilon>0$  depending only on $\epsilon_0$, $\alpha_0$ and  $n$ such that
\begin{align*}\label{G-v}
 |G_{N}(v, T(\Omega)) \cap  U_{\alpha_0}|
 &\geq \left\{1-\epsilon_0 -C N^{-\tau} \Big(\fint_{U} |\tilde f|^n ~dy\Big)^{\frac{\tau}{n}} \right\}~ |U_{\alpha_0}|\\
 &=\left\{1-\epsilon_0 -C
 \big(\frac{\alpha_0}{\theta \gamma N}\big)^\tau \big(t_0 |\det T|^{\frac{2}{n}}\big)^{-\tau}
  \Big(\fint_{S_\phi(x_0,   \frac{t_0}{\alpha_0})} | f|^n ~dx\Big)^{\frac{\tau}{n}} \right\}~ |U_{\alpha_0}|\nonumber\\
 &\geq \left\{1-\epsilon_0 -C
  \big(\frac{\alpha_0}{\theta \gamma N}\big)^\tau \Big(\fint_{S_\phi(x_0,   \frac{t_0}{\alpha_0})} | f|^n ~dx\Big)^{\frac{\tau}{n}} \right\}~ |U_{\alpha_0}|\nonumber.
 \end{align*}
But  since $U_{\alpha_0} =T(S_\phi(x_0,  t_0))$ and  $v(y) =\frac{\alpha_0}{\theta \gamma t_0} \big[u(T^{-1}y) - u(\bar x) -\langle
 \nabla u(\bar x), T^{-1}y -\bar x \rangle \big]$,   the same calculations leading to \eqref{G-transformation} yield
\begin{align*}
&G_{N}(v, T(\Omega)) \cap U_{\alpha_0}
=T \Big( G_{\frac{ N\theta \gamma t_0 |\det T|^{\frac{2}{n}}}{\alpha_0} }(u,\Omega) \cap S_\phi(x_0,   t_0)\Big)
\approx  T \Big( G_{\frac{ N\theta \gamma}{\alpha_0} }(u,\Omega) \cap S_\phi(x_0,   t_0)\Big).
\end{align*}
Therefore  we obtain
\begin{align*}
 \left| T \Big( G_{\frac{ N\theta \gamma}{\alpha_0} }(u,\Omega) \cap S_\phi(x_0,   t_0)\Big) \right|
 &\geq \left\{1-\epsilon_0 -C
  \big(\frac{\alpha_0}{\theta \gamma N}\big)^\tau
   \Big(\fint_{S_\phi(x_0,  \frac{t_0}{\alpha_0})} | f|^n ~dx\Big)^{\frac{\tau}{n}} \right\}~ \left|T( S_\phi(x_0,  t_0))\right|.
 \end{align*}
By setting $N'= N\theta/\alpha_0$, we can rewrite this as 
\begin{align*}
 \left| G_{ N' \gamma}(u,\Omega) \cap S_\phi(x_0,  t_0)\right|
 &\geq \left\{1-\epsilon_0 -C
  \big(\frac{1}{ \gamma N'}\big)^\tau
 \Big(\fint_{S_\phi(x_0,  \frac{t_0}{\alpha_0})} | f|^n ~dx\Big)^{\frac{\tau}{n}} \right\}~ \left|S_\phi(x_0, t_0)\right|
 \end{align*}
for any $N'\geq N_0=N_0(\alpha_0,n)$.

\end{proof}

\subsection{$W^{2,p}$ estimate}
In this subsection we will use the density estimates established in Subsection~\ref{sub:improveddensity} to derive 
interior $W^{2,p}$-estimates for solution $u$ of the linearized equation $\calL_\phi u=f$ when $f\in L^q(\Omega)$ for some $q>n$.   We begin with the following key result which gives a solution to the conjecture  in \cite{GT}.
 
\begin{theorem}\label{premainhomo}
Let $\Omega$ be a normalized convex domain and $u\in W^{2,n}_{loc}(\Omega)$ be a solution of
$\calL_\phi u=f$ in $\Omega$, where  
 $\phi\in C(\overline\Omega)$ is a convex function satisfying $\phi=0$ on $\partial \Omega$. 
 Let $p>1$, $\max{\{n,p \}}<q <\infty$ and let $0<\alpha<1$. Then there exist positive constants $\epsilon$ and $C$ depending only on $p$, 
 $q$,  $\alpha$  and $n$ such that
 if  $1-\epsilon \leq \det D^2\phi \leq 1+\epsilon$, we have
\begin{equation*}
\|D^2u\|_{L^p(\Omega_{\alpha})}\leq
C\Big( \|u\|_{L^\infty(\Omega)} +\|f\|_{L^{q}(\Omega)}\Big).
\end{equation*}
\end{theorem}
\begin{proof}
We first observe that by working with the function  $v :=\dfrac{\epsilon u}{\epsilon \|u\|_{L^\infty(\Omega)} +\|f\|_{L^{q}(\Omega)}}$ instead of $u$, it is enough to show that  there exist $\epsilon, C>0$  depending only on $p$, $q$, $\alpha$ and $n$ such that
 if  $1-\epsilon \leq \det D^2\phi \leq 1+\epsilon$, $\|u\|_{L^\infty(\Omega)}\leq 1$ and $\|f\|_{L^{q}(\Omega)}\leq \epsilon$, then  
\begin{equation}\label{reduction}
\|D^2u\|_{L^p(\Omega_{\alpha})}\leq C.
\end{equation}
Note also that $u\in C^1(\Omega)$ as a consequence of $C^{1,\alpha}_{loc}$ estimates in \cite[Theorem~4.5]{GN}.

Let $\alpha_0 :=\frac{\alpha+1}{2}$ and $N_0 =N_0(\alpha_0,n)$ be the largest of the constants in Lemma~\ref{lm:improved-density-I} and
Lemma~\ref{lm:furthercriticaldensitytwo}.
Fix $M\geq N_0$ so that
$\frac{1}{C_0^{1/\gamma} (M^{1/\gamma} -1)}\leq \frac{1-\alpha_0}{2}$ and $(c M^{\frac{q-p}{2p}})^{\frac{1}{n-1}}\geq \frac{\diam(\Omega)}{\sqrt{\eta_0}}$, where $\gamma$, $C_0$ are given by  Lemma~\ref{containment2} and $c$ is given by Lemma~\ref{distribution-function} when $\lambda=1/2$ and $\Lambda=3/2$. 
Next select $0<\epsilon_0<1/2$ such that
\[
M^{q} \sqrt{2 \epsilon_0}  =\frac{1}{2}
\]
and  $\epsilon=\epsilon(\epsilon_0,\alpha_0, n)=\epsilon(p,q,\alpha,n)$ be the smallest  of the constants in
Lemma~\ref{lm:improved-density-I} and
Lemma~\ref{lm:furthercriticaldensitytwo}.
With this choice of $\epsilon$, we are going to show that \eqref{reduction} holds.
Applying Lemma~\ref{lm:improved-density-I} to the function
$u$ and using  $\|f\|_{L^{q}(\Omega)}\leq \epsilon$ we obtain 
\begin{align*}
\big|S_\phi(x_0,t_0)\cap G_{\frac{M}{t_0}}(u,\Omega)\big|
\geq
\left(1-\epsilon_0 - C \epsilon^\tau \right)\, |S_\phi(x_0,t_0)|
\end{align*}
as long as $S_\phi(x_0,\frac{t_0}{\alpha_0})\subset \Omega_{\frac{\alpha_0 +1}{2}}$, where $C=C(p,\alpha,n)$ and $\tau=\tau(\alpha,n)$.
By taking $\epsilon$ even smaller if necessary we can assume $C \epsilon^\tau <\epsilon_0$. Then it follows from the above inequality that
\begin{equation}\label{eq:consequence-lemma4.3}
\big|S_\phi(x_0,t_0)\setminus G_{\frac{M}{t_0}}(u,\Omega)\big| 
\leq 2\epsilon_0 \, |S_\phi(x_0,t_0)|
\quad\mbox{for any}\quad S_\phi(x_0,\frac{t_0}{\alpha_0})\subset \Omega_{\frac{\alpha_0 +1}{2}}.
\end{equation}

Let $\eta(\alpha_0)>0$ be given by Lemma~\ref{lm:furthercriticaldensitytwo} ensuring in particular that $S_\phi(x, \frac{t}{\alpha_0})\subset \Omega_{\frac{\alpha_0 +1}{2}}$ for all $x\in\Omega_{\alpha_0}$ and $t\leq \eta(\alpha_0)$.  Assume 
$\alpha_2<\alpha_1<\alpha_0$ are such that there exist $\eta_2<\eta_1\leq \eta(\alpha_0)$ with the property: 
if $x\in \Omega_{\alpha_2}$ and $t\leq \eta_2$ then
$S_\phi(x,t)\subset \Omega_{\alpha_1}$; and if $x\in \Omega_{\alpha_1}$ and $t\leq \eta_1$ then
$S_\phi(x,t)\subset \Omega_{\alpha_0}$. 
With these choices and
for $1/h \leq \eta_2$, by using \eqref{eq:consequence-lemma4.3} and the same arguments leading to \eqref{eq:criticaldensitynonhomoge} we obtain:
%
for any $x_0\in
\Omega_{\alpha_2}\setminus G_{h M}(u,\Omega)$ there is
$t_{x_0}\leq 1/h$ such that
\begin{equation}\label{eq:generalcriticaldensity}
\big|\big(\Omega_{\alpha_2}\setminus G_{h M}(u,\Omega)\big) \cap
S_\phi(x_0,t_{x_0})\big| = 2\epsilon_0 \, |S_\phi(x_0,t_{x_0})|.
\end{equation}
We now claim that \eqref{eq:generalcriticaldensity} implies 
\begin{equation}\label{eq:section-contained}
S_\phi(x_0,t_{x_0}) \subset \big(\Omega_{\alpha_1}\setminus G_h
(u,\Omega)\big) \cup \big\{x\in \Omega_{\alpha_0}: \mathcal M_\mu ((f/\det D^2
\phi)^n)(x)> (c^* M h)^n\big\},
\end{equation}
where $c^* :=(\frac{\epsilon_0}{C})^{1/\tau}$ and $\mu := M\phi$.
Otherwise, and since $x_0\in \Omega_{\alpha_2}$ and $t_{x_0}\leq
1/h \leq \eta_2$, we have that $S_\phi(x_0,t_{x_0})\subset
\Omega_{\alpha_1}$ and there exists $\bar x\in
S_\phi(x_0,t_{x_0})\cap G_h(u,\Omega)$ such that $\mathcal M_\mu
((f/\det D^2 \phi)^n)(\bar x)\leq (c^* M h)^n$. Note also that  $ t_{x_0} \leq   \eta(\alpha_0)$ and due to our assumption on $\phi$   the measure  $\mu$ is comparable to the Lebesgue  measure. Then by
Lemma~\ref{lm:furthercriticaldensitytwo} applied to $u$ we get 
\begin{align*}
\big|S_\phi(x_0,t_{x_0})\cap G_{h M}(u,\Omega)\big| > (1-2 \epsilon_0) \, |S_\phi(x_0,t_{x_0})|
\end{align*}
yielding
\[
\big|\big(\Omega_{\alpha_2}\setminus G_{h M}(u,\Omega)\big)\cap
S_\phi(x_0,t_{x_0})\big| \leq \big|S_\phi(x_0,t_{x_0})\setminus
G_{h M}(u,\Omega)\big| < 2\epsilon_0 \,|S_\phi(x_0,t_{x_0})|.
\]
This is a contradiction with \eqref{eq:generalcriticaldensity} and so
\eqref{eq:section-contained} is proved. We infer from \eqref{eq:generalcriticaldensity}, 
\eqref{eq:section-contained}
and  \cite[Theorem~6.3.3]{G} that
\begin{align}\label{eq:generalpowerdecay}
&|\Omega_{\alpha_2}\setminus G_{h M}(u,\Omega)|\\
&\leq  \sqrt{2\epsilon_0}\left[
|\Omega_{\alpha_1}\setminus G_{h}(u,\Omega)|
 +  \big|\{x\in \Omega_{\alpha_0}:
\mathcal M_\mu ((f/\det D^2\phi)^n)(x)> (c^* M h)^n\}\big|\right],\nonumber
\end{align}
as long as $\alpha_2<\alpha_1<\alpha_0$ are such that $\eta_2<\eta_1\leq \eta(\alpha_0)$, and $1/h \leq
\eta_2$.

For $k=0,1,\dots$, set
\[
a_k:= |\Omega_{\alpha_k}\setminus G_{M^k}(u,\Omega)| \quad \text{and}\quad b_k:=\big|\{x\in \Omega_{\alpha_0}:
\mathcal M_\mu ((f/\det D^2\phi)^n)(x)> (c^* M M^k)^n\}\big|, 
\]
where $\alpha_k$ will be defined inductively in the sequel. First fix $\alpha_1$ so that 
\[
\frac{3\alpha_0-1}{2}  <\alpha_1<\alpha_0\quad\mbox{and}\quad \eta_1 :=C_0 (\alpha_0-\alpha_1)^{\gamma}\leq \eta(\alpha_0).
\]
By taking  $M$ even larger if necessary, we can assume that $1/M< \eta_1$. Let $h=M$, and set
$\alpha_2=\alpha_1 -\frac{1}{(C_0 M)^{1/\gamma} }$. Then 
$\frac{1}{h}
=\frac{1}{ M}=C_0 (\alpha_1-\alpha_2)^{\gamma}=: \eta_2$, and
so from  Lemma~\ref{containment2} and \eqref{eq:generalpowerdecay} we get
$a_2\leq \sqrt{2\epsilon_0}(a_1 + b_1)$.
Next let $h=M^2$ and $\alpha_3 = \alpha_2 -\frac{1}{(C_0 M^2)^{1/\gamma} }$, so
$\frac{1}{h}=C_0
(\alpha_2-\alpha_3)^{\gamma}=: \eta_3$. Then
$a_3\leq \sqrt{2\epsilon_0}(a_2 + b_2)
\leq 2\epsilon_0 a_1 +  2\epsilon_0 b_1 + \sqrt{2\epsilon_0}\, b_2$.
Continuing in this way we conclude that
\begin{align*}
a_{k+1}\leq (\sqrt{2\epsilon_0})^k a_1 + \sum_{i=1}^{k}{(\sqrt{2\epsilon_0})^{(k+1)-i} b_i}\quad \mbox{for}\quad k=1,2,\dots
\end{align*}
On the other hand, $\alpha_{k+1}=\alpha_1- \sum_{j=1}^k \frac{1}{(C_0 M^j)^{1/\gamma}} \geq \frac{3\alpha_0-1}{2} - \frac{1}{C_0^{1/\gamma} (M^{1/\gamma} -1)} 
\geq 2\alpha_0 -1=\alpha$ by our choice of $\alpha_1$, $M$ and $\alpha_0$.
Therefore for every $k\geq 1$, 
\begin{equation}\label{improved-est}
|\Omega_{\alpha}\setminus G_{M^{k+1}}(u,\Omega)|
\leq
|\Omega_{\alpha_{k+1}}\setminus G_{M^{k+1}}(u,\Omega)|\leq (\sqrt{2\epsilon_0})^k a_1 + \sum_{i=1}^{k}{(\sqrt{2\epsilon_0})^{(k+1)-i} b_i}.
\end{equation}

Next let $\Theta(u)$ be the function  defined by \eqref{function-theta}. We claim that $\Theta(u)\in L^p(\Omega_{\alpha})$ and
\begin{equation}\label{Theta-est}
\|\Theta(u)\|_{L^p(\Omega_{\alpha})}\leq C(p,q,\alpha,n).
\end{equation}
Indeed since $u\in C^1(\Omega)$, it is easy to see that $\Theta(u)$ is lower semicontinuous in $\Omega_\alpha$ and so   measurable there. Moreover, we have 
\begin{align*}\label{Lp-distribution-fn}
&\int_{\Omega_{\alpha}}{|\Theta(u)|^p ~dx}
= p\int_0^{\infty}{t^{p-1} \big|\{x\in \Omega_{\alpha}: \,\Theta(u)(x) >t \}\big| ~dt}\\
&= p \int_0^{M^{\frac{q}{p}}}{t^{p-1} \big|\{x\in \Omega_{\alpha}: \Theta(u)(x) >t \}\big| ~dt}
+ p\sum_{k=1}^{\infty}\int_{M^{\frac{qk}{p}}}^{M^{\frac{q(k+1)}{p}}}{t^{p-1} \big|\{x\in \Omega_{\alpha}: \Theta(u)(x) >t \}\big| ~dt}\nonumber\\
&\leq |\Omega_{\alpha}| M^{q} +\big(M^{q} -1 \big)  \sum_{k=1}^{\infty}{M^{qk} \big|\{x\in \Omega_{\alpha}: ~\Theta(u)(x) > M^{\frac{qk}{p}} \}\big|} \nonumber\\
&\leq |\Omega_{\alpha}| M^{q}
+\big(M^{q} -1 \big) \left[ \sum_{k=1}^{\infty}{M^{qk} \big|\Omega_{\alpha}\setminus D^{\alpha}_{(c M^{\frac{k(q-p)}{2p}})^{\frac{1}{n-1}}}\big|}  +  \sum_{k=1}^{\infty}{M^{qk} \big|\Omega_{\alpha}\setminus G_{M^k}(u,\Omega)\big|}\right]\nonumber\\
&\leq |\Omega_{\alpha}| M^{q}
+\big(M^{q} -1 \big) \left[ \frac{|\Omega|}{(C_n\epsilon)^2} c^{\ln{\sqrt{C_n \epsilon}}} \sum_{k=1}^{\infty}{ M^{k\left(q + (\frac{q}{p}-1)\frac{\ln{\sqrt{C_n \epsilon}}}{C}\right)}}   + \sum_{k=1}^{\infty}{M^{qk} \big|\Omega_{\alpha}\setminus G_{M^k}(u,\Omega)\big|}\right],\nonumber
\end{align*}
where we used \eqref{eq:distributionsetofsecondderivatives}
with $\kappa =q/p>1$ and $\beta =M^k$ in the second inequality and used \eqref{eq:exponentialdecayofthesetD} in the last inequality.
Since $\epsilon>0$ is small, the first summation in the last expression  is finite
and hence \eqref{Theta-est}  will follow if we can show that $\sum_{k=1}^{\infty}{M^{kq} |\Omega_{\alpha}\setminus G_{M^k}(u,\Omega)|}\leq C$. For this, let us employ 
 \eqref{improved-est} to obtain
\begin{align*}
\sum_{k=1}^{\infty}{M^{kq} |\Omega_{\alpha}\setminus G_{M^k}(u,\Omega)|}
&\leq a_1 \sum_{k=1}^{\infty} {M^{kq} (\sqrt{2\epsilon_0})^{k-1} } + \sum_{k=1}^{\infty} \sum_{i=0}^{k-1}{M^{kq} (\sqrt{2\epsilon_0})^{k-i}  b_i}\\
&=\frac{ a_1}{\sqrt{2\epsilon_0}} \sum_{k=1}^{\infty} {\big(M^{q} \sqrt{2\epsilon_0}\big)^{k} } + \sum_{i=0}^{\infty} \sum_{k=i+1}^{\infty}{M^{(k-i)q} (\sqrt{2\epsilon_0})^{k-i} M^{iq} b_i}\\
&=\frac{ a_1}{\sqrt{2\epsilon_0}} \sum_{k=1}^{\infty} {\big(M^{q} \sqrt{2\epsilon_0}\big)^{k} } +
\Big[\sum_{j=1}^{\infty}{\big(M^{q} \sqrt{2\epsilon_0}\big)^j}\Big]
 \Big[\sum_{i=0}^{\infty} { M^{iq} b_i}\Big]\\
 &=\frac{ a_1}{\sqrt{2\epsilon_0}} \sum_{k=1}^{\infty} {2^{-k} } +
\Big[\sum_{j=1}^{\infty}{2^{-j}}\Big]
 \Big[\sum_{i=0}^{\infty} { M^{iq} b_i}\Big]
= \frac{ a_1}{\sqrt{2\epsilon_0}}  +
 \sum_{i=0}^{\infty} { M^{iq} b_i}.
\end{align*}
But as $f^n \in L^{\frac{q}{n}}(\Omega)$ and $q>n$, we have from  Theorem
\ref{strongtype} that
\begin{align*}
\int_{\Omega_{\alpha_0}}{\Big|\mathcal M_\mu \Big(\big(\frac{f}{\det D^2\phi}\big)^n\Big)(x) \Big|^{\frac{q}{n}} ~d\mu(x)}
\leq C_n\, \int_\Omega \Big|\frac{f}{\det
D^2 \phi}\Big|^{q}\, d\mu(y) \leq C
\end{align*}
implying
$ \sum_{i=0}^{\infty} { (M^n)^{i \frac{q}{n}} b_i}\leq C$.
Thus   $\sum_{k=1}^{\infty}{M^{kq} |\Omega_{\alpha}\setminus G_{M^k}(u,\Omega)|}\leq C$  and claim \eqref{Theta-est} is proved.

It follows from \eqref{Theta-est} and \cite[Proposition 1.1]{CC} that $D^2u\in L^p(\Omega_\alpha)$ and $\|D^2 u\|_{L^p(\Omega_{\alpha})}\leq 4\|\Theta(u)\|_{L^p(\Omega_{\alpha})}\leq C(p,q,\alpha,n)$. This gives
 \eqref{reduction} as desired and the proof is complete.
\end{proof}

We are finally in a position to prove the main result of the paper, Theorem~\ref{mainhomo}.

\begin{proof}[Proof of Theorem~\ref{mainhomo}]
Let $\epsilon=\epsilon(p,q,n)$ be the constant given by Theorem~\ref{premainhomo} corresponding to $\alpha=1/2$. Let $x\in\Omega'$ and suppose a section $S=S_\phi(x,\delta)\Subset\Omega$ is such that $|g(z)-g(x)|\leq \lambda \epsilon$, for each $z\in S$.
Then by the property of  sections \cite[Theorem~3.3.8]{G}, we have
\begin{equation}\label{eq:sectionisbetweenballs}
B(x, K_1 \delta)\subset S\subset B(x, K_2 \delta^b),
\end{equation}
with  $K_1, K_2, b$ positive constants depending only on $\lambda, \Lambda$ and $n$. 
Let $Tx=A x +b$ be an affine map normalizing $S$ and consider the following functions on $\tilde \Omega:=T(S)$:
\begin{align*}
&\tilde \phi(y) := \frac{|\det A|^{\frac{2}{n}}}{g(x)^{\frac{1}{n}}}\left[ \phi(T^{-1} y) - \phi(x)
- \nabla \phi(x)\cdot (T^{-1} y-x) -\delta\right],\\
&\text{and}\quad  \tilde u(y):= |\det A|^{\frac{2}{n}} g(x)^{\frac{n-1}{n}} u(T^{-1}y).
\end{align*}
We have $\tilde \Omega$
is normalized, $\tilde\phi=0$ on $\partial \tilde \Omega$ and
$D^2\tilde \phi(y)= \frac{|\det A|^{\frac{2}{n}}}{g(x)^{\frac{1}{n}}}\, (A^{-1})^t\,
D^2\phi(T^{-1}y) \, A^{-1}$.
Thus $\det D^2 \tilde\phi(y)=\frac{g(T^{-1} y)}{g(x)}=:\tilde g(y)$ and if $\tilde \Phi(y)$ is  the cofactor matrix of $D^2\tilde
\phi(y)$, then
\[
\calL_{\tilde \phi}\tilde u (y)=\trace (\tilde \Phi(y) \,D^2\tilde
u(y))=\trace \big(\Phi(T^{-1}y) \,D^2 u(T^{-1} y)\big)=f(T^{-1} y)=:\tilde f(y) \quad\mbox{in }\tilde \Omega.
\]
Moreover since $g(x)-\lambda\epsilon \leq g(z)\leq g(x) +\lambda \epsilon$ for $z\in S$ and $g\geq \lambda$, we get
\[
1-\eps\leq
1-\frac{\epsilon\lambda}{g(x)}\leq\tilde g(y)\leq 1 +\frac{\lambda\epsilon}{g(x)}\leq 1 + \eps\quad \mbox{for}\quad y\in\tilde\Omega. 
\]
Therefore, we can  apply Theorem~\ref{premainhomo} to obtain
\begin{align}\label{eq:localized-Lp-est}
\left(\int_{\tilde\Omega_{1/2}}{|D^2\tilde u(y)|^p ~dy}\right)^{1/p} 
&\leq C(p,q,n)
\Big(\|\tilde u\|_{L^\infty(\tilde \Omega)}
+\|\tilde f\|_{L^{q}(\tilde\Omega)} \Big)\\
&=C 
\Big(|\det A|^{\frac{2}{n}} g(x)^{\frac{n-1}{n}} \| u\|_{L^\infty(S)}
+|\det A|^{\frac{1}{q}} \|f\|_{L^{q}(S)} \Big).\nonumber
\end{align}
By the definition of $\tilde u$ we have
$ D^2u(z)= |\det A|^{\frac{-2}{n}} g(x)^{\frac{1-n}{n}} A^t\,D^2\tilde u(T z)\, A $ in $ S$, and consequently  $\|D^2u\|_{L^p(S_{1/2})}
\leq \|A\|^2 \, |\det A|^{-(\frac{2}{n} +\frac{1}{p})} g(x)^{\frac{1-n}{n}} \|D^2 \tilde u\|_{L^p(\tilde\Omega_{1/2})}$ where $S_{1/2}:= S_\phi(x, \delta/2)$.
Notice that $|\det A| \approx \delta^{-n/2}$ by the normalization, and $\|A\|\leq C \delta^{-1}$ by the fact $A B(x, K_1 \delta) + b \subset B_n(0)$ following from \eqref{eq:sectionisbetweenballs}.
Hence we deduce from \eqref{eq:localized-Lp-est} that
\begin{align}\label{eq:L^pestimateonsections}
\left(\int_{S_{1/2}}{|D^2 u(z)|^p \,dz}\right)^{1/p} 
&\leq C \|A\|^2 |\det A|^{\frac{-1}{p}}
\Big( \| u\|_{L^\infty(S)}
+|\det A|^{\frac{1}{q} -\frac{2}{n}} g(x)^{\frac{1-n}{n}} \|f\|_{L^{q}(S)} \Big)\\
&\leq C \delta^{\frac{n}{2 p} -2}\|u\|_{L^\infty(\Omega)}
+ C \delta^{\frac{n}{2 p}-\frac{n}{2 q} -1}\|f\|_{L^{q}(\Omega)},\nonumber
\end{align}
where $C$ depends only on $p$, $q$,  $\lambda$, $\Lambda$ and $n$.
 
Now since $\Omega'\Subset\Omega$, we can pick $\delta$ small depending only on the parameters $\lambda, \Lambda, n, \dist(\Omega',\partial\Omega)$ and the modulus of continuity of $g$ such that for each $x\in \Omega'$ we have $B(x, K_2 \delta^b)\Subset\Omega$ and
$|g(z)-g(x)|\leq \lambda \epsilon$ in $B(x, K_2 \delta^b)$. Next select a finite covering of $\Omega'$ by balls $\{B(x_j, K_1 \delta)\}_{j=1}^N$ with $x_j\in \Omega'$, then the  desired inequality follows by adding \eqref{eq:L^pestimateonsections} over $\{S_\phi(x_j, \delta/2)\}_{j=1}^N$. 
\end{proof}

In this paper we have chosen to work with strong solutions in $W^{2,n}_{loc}(\Omega)$ in order to reveal direct calculations. However, the interior $W^{2,p}$ estimates in Theorem~\ref{premainhomo} and Theorem~\ref{mainhomo} can be derived for viscosity solutions of $\calL_\phi u=f$ by modifying slightly the definition of the set $G_M(u,\Omega)$ and following our arguments. For this purpose we note that the Caffarelli-Guti\'errez interior H\"older estimates, which were used in Lemma~\ref{explest}, still hold for viscosity solutions as observed by Trudinger and Wang in \cite{TW4}.

\bibliographystyle{plain}

\begin{thebibliography}{\textbf{TW97}}

\bibitem[\textbf{C1}]{C1} L. A.~Caffarelli. {\em Interior a priori estimates for solutions of fully nonlinear equations.} Ann. of Math. {\bf 130},  no. 1, 189--213, 1989.


\bibitem[\textbf{C2}]{C2} L. A.~Caffarelli. {\em A localization property of viscosity solutions to the Monge-Amp\`ere equation and their strict convexity.}
    Ann. of Math. {\bf 131}, 129--134, 1990.
    
\bibitem[\textbf{C3}]{C3} L. A.~Caffarelli.  {\em Interior $W^{2,p}$
estimates for solutions to the Monge-Amp\`ere equation.}
    Ann. of Math. {\bf 131}, 135--150, 1990.

\bibitem[\textbf{C4}]{C4} L. A.~Caffarelli. {\em Some regularity properties of
    solutions of Monge-Amp\`ere equation.} Comm. Pure Appl. Math. {\bf 44},
    965--969, 1991.

\bibitem[\textbf{CC}]{CC} L. A.~Caffarelli and X.~Cabr\'e. {\em Fully nonlinear elliptic
equations.} American Mathematical Society Colloquium Publications,
volume 43, 1995.

\bibitem[\textbf{CG1}]{CG1} L. A.~Caffarelli and C. E.~Guti\'errez.
{\em Real analysis related to the
    Monge-Amp\`ere equation.} Trans. Amer. Math. Soc. {\bf 348}, no. 3,     1075--1092, 1996.

\bibitem[\textbf{CG2}]{CG2} L. A.~Caffarelli  and C. E.~Guti\'errez. {\em Properties of the solutions of the
linearized Monge--Amp\`ere equation.} Amer. J. Math. {\bf
119}, no. 2, 423--465, 1997.

\bibitem[\textbf{CNP}]{CNP}
M. J.~Cullen, J.~Norbury and R. J.~Purser.  {\em Generalized Lagrangian solutions for atmospheric and oceanic
 flows.}
 SIAM J. Appl. Math.  {\bf 51},  no. 1, 20--31, 1991.

\bibitem[\textbf{dPF}]{dPF}
G.~De Phillips and A.~Figalli. 
{\em $W^{2,1}$ regularity for solutions of the Monge--Amp\`ere equation.} Preprint, 2011.

\bibitem[\textbf{D1}]{D1}
S. K.~Donaldson.  
{\em Scalar curvature and stability of toric varieties.}
 J. Differential Geom. {\bf 62},
no. 2, 289–-349, 2002.

\bibitem[\textbf{D2}]{D2} S. K.~Donaldson.  {\em  Interior estimates for solutions of Abreu's equation.}
 Collect. Math.  {\bf 56},  no. 2, 103--142, 2005.
 
 \bibitem[\textbf{D3}]{D3} S. K.~Donaldson.  
 {\em Extremal metrics on toric surfaces: a continuity method.}
  J. Differential Geom. {\bf 79}, no. 3, 389–-432, 2008.

\bibitem[\textbf{D4}]{D4} S. K.~Donaldson.  
{\em Constant scalar curvature metrics on toric surfaces.} Geom. Funct. Anal. {\bf 19}, no. 1, 83-–136, 2009.

\bibitem[\textbf{Es}]{Es} L.~Escauriaza.
{\em $W^{2,n}$ a priori estimates for solutions to fully non-linear equations.} Indiana University Mathematics Journal {\bf 42}, no. 2, 413--423, 1993.

\bibitem[\textbf{E}]{E} L. C.~Evans.
{\em Some estimates for non divergence structure, second order
elliptic equations.} Trans. Amer. Math. Soc. {\bf 287}, no. 2,  701--712, 1985.

\bibitem[\textbf{GiT}]{GiT} D.~Gilbarg and N.S.~Trudinger.
{\em Elliptic partial differential equations of second order.}
Springer--Verlag, New York, 2001.


  \bibitem[\textbf{G}]{G} 
  C. E.~Guti\'errez. {\em The Monge-Amp\`ere equation.}
  Birkh\"auser, Boston, 2001.

\bibitem[\textbf{GH}]{GH}
C. E.~Guti\'errez and Q.~Huang. {\em Geometric properties
of the sections of solutions to the Monge--Amp\`ere equation.}
Trans. Amer. Math. Soc. {\bf 352}, 4381--4396, 2000.

\bibitem[\textbf{GN}]{GN}
C. E.~Guti\'errez and T.V.~Nguyen.
{\em Interior gradient estimates for solutions to the linearized Monge-Amp\`ere equation}.
Adv. Math., {\bf 228}, 2034--2070, 2011. 

\bibitem[\textbf{GT}]{GT}
C. E.~Guti\'errez and  F.~Tournier. 
{\em $W^{2,p}$-estimates for the linearized Monge-Amp\`ere equation}.
 Trans. Amer. Math. Soc.  {\bf 358},  no. 11, 4843--4872, 2006.
 
\bibitem[\textbf{H}]{H} Q.~Huang. {\em Sharp regularity results on second derivatives of solutions to the
 Monge-Amp\`ere equation with VMO type data.}
 Comm. Pure Appl. Math.  {\bf 62},  no. 5, 677--705, 2009.
 
 \bibitem[\textbf{LS1}]{LS1}
 N. Q.~Le and O.~Savin.
 {\em Boundary regularity for solutions to the linearized Monge-Amp\`ere equations.} Preprint, 2011.
 
 \bibitem[\textbf{LS2}]{LS2}
 N. Q.~Le and O.~Savin.
 {\em Some minimization problems in the class of convex functions with prescribed determinant.} Preprint, 2011.
 
\bibitem[\textbf{L}]{L} F.-H.~Lin.
{\em Second derivative $L^p$--estimates for elliptic equations of
nondivergent type.} Proc. Amer. Math. Soc. {\bf 96}, no. 3, 447--451, 1986.

\bibitem[\textbf{Lo}]{Lo}
G.~Loeper. {\em  A fully nonlinear version of the incompressible Euler equations: the
 semigeostrophic system.}
 SIAM J. Math. Anal.  {\bf 38},  no. 3, 795--823, 2006.




\bibitem[\textbf{PT}]{PT}
C.~Pucci and G.~Talenti. {\em Elliptic (second-order) partial differential equations with measurable
 coefficients and approximating integral equations.}
Adv. Math.,  {\bf 19},  no. 1, 48--105, 1976.



\bibitem[\textbf{S1}]{S1} O.~Savin. {\em A Liouville theorem for solutions to the linearized Monge-Ampere
 equation.}
 Discrete Contin. Dyn. Syst.  {\bf 28},  no. 3, 865--873, 2010.

\bibitem[\textbf{S2}]{S2} O.~Savin. {\em
Global $W^{2,p}$ estimates for the Monge-Ampere equations.} Preprint, 2010.

\bibitem[\textbf{Sc}]{Sc} R.~Schneider.  {\em Convex bodies:  The Brunn-Minkowski theory.} Cambridge University Press, Cambridge, 1993.


\bibitem[\textbf{St}]{St} E.M.~Stein.  {\em Singular integrals and differentiability properties of functions.}
 Princeton University Press, Princeton, N.J.,  1970.


\bibitem[\textbf{Sw}]{Sw} A.~\'Swiech.
{\em $W\sp {1,p}$-interior estimates for solutions of fully nonlinear, uniformly elliptic equations.}
 Adv. Differential Equations  {\bf 2},  no. 6, 1005--1027, 1997.
 
 
\bibitem[\textbf{TiW}]{TiW} G.-J.~Tian and X.-J.~Wang.
{\em A class of Sobolev type inequalities.} Methods and 
Applications of Analysis {\bf 15}, no. 2, 263--276, 2008.

\bibitem[\textbf{T}]{T} N.~Trudinger.
{\em Glimpses of nonlinear partial differential equations in the twentieth
 century: A priori estimates and the Bernstein problem.}
 Challenges for the 21st century (Singapore, 2000), 
 196--212, World Sci. Publ., River Edge, NJ,  2001. 



\bibitem[\textbf{TW1}]{TW1} N.~Trudinger and X.-J.~Wang.
{\em The Bernstein problem for affine maximal hypersurfaces.}
Invent. Math. {\bf 140}, no. 2, 399--422, 2000.


\bibitem[\textbf{TW2}]{TW2} 
N.~Trudinger and X.-J.~Wang.
{\em The affine plateau problem.} J. Amer. Math. Soc. {\bf 18}, 253--289, 2005.

\bibitem[\textbf{TW3}]{TW3} 
N.~Trudinger and X.-J.~Wang.
{\em Boundary regularity for Monge-Amp\`ere and affine maximal surface
equations.} Ann. of Math. {\bf 167}, 993--1028, 2008.

\bibitem[\textbf{TW4}]{TW4}
N.~Trudinger and X.-J.~Wang.
{\em The Monge-Amp\`ere equation and its geometric applications.}
Handbook of geometric analysis. No. 1, 
 Adv. Lect. Math., 7, Int. Press, Somerville, MA, 467--524, 2008. 

\bibitem[\textbf{U}]{U}
N.~Uralt'ceva.  
{\em The impossibility of $W^{2,q}$ estimates for multidimensional
 elliptic equations with discontinuous coefficients.}
(Russian)  Zap. Naučn. Sem. Leningrad. Otdel. Mat. Inst. Steklov. (LOMI)  {\bf 5},  250--254, 1967.



\bibitem[\textbf{WL}]{WL} L.~Wang.
{\em A geometric approach to the Calder\'on-Zygmund estimates.} Acta Math. Sin. (Engl. Ser.) {\bf 19}, no. 2, 381--396, 2003.


\bibitem[\textbf{W}]{W} X.-J.~Wang.
{\em Some counterexamples to the regularity of Monge--Amp\`ere
equations.} Proc. Amer. Math. Soc. {\bf 123}, no. 3, 841--845, 1995.

\bibitem[\textbf{Z1}]{Z1} B.~Zhou.
{\em The first boundary value problem for Abreu's equation.} Preprint, 2010.

\bibitem[\textbf{Z2}]{Z2} B.~Zhou.
{\em The Bernstein theorem for a class of fourth order
equations.} Calc. Var. Partial Differential Equations, doi 10.1007/s00526-011-0401-3.
\end{thebibliography}

\end{document}